\documentclass[english]{elsarticle}
\usepackage[T1]{fontenc}
\usepackage[latin9]{inputenc}
\usepackage{geometry}
\usepackage{xcolor}
\usepackage{pdfcolmk}
\usepackage{babel}
\usepackage{amsthm}
\usepackage{amstext}
\usepackage{amssymb}
\usepackage{mathdots}
\usepackage{graphicx}
\PassOptionsToPackage{normalem}{ulem}
\usepackage{ulem}
\usepackage[unicode=true,
 bookmarks=true,bookmarksnumbered=false,bookmarksopen=false,
 breaklinks=false,pdfborder={0 0 1},backref=false,colorlinks=false]
 {hyperref}

\makeatletter

\newcommand{\lyxdot}{.}

\providecolor{lyxadded}{rgb}{0,0,1}
\providecolor{lyxdeleted}{rgb}{1,0,0}

  \theoremstyle{plain}
  \newtheorem*{thm*}{\protect\theoremname}
\theoremstyle{plain}
\newtheorem{thm}{\protect\theoremname}[section]
  \theoremstyle{plain}
  \newtheorem{cor}[thm]{\protect\corollaryname}
  \theoremstyle{plain}
  \newtheorem{prop}[thm]{\protect\propositionname}

\makeatletter
\let\eqref\relax
\usepackage{amsmath}
\usepackage{mathtools}

\usepackage{bbm}

\@ifundefined{showcaptionsetup}{}{%
 \PassOptionsToPackage{caption=false}{subfig}}
\usepackage{subfig}
\makeatother

  \providecommand{\corollaryname}{Corollary}
  \providecommand{\propositionname}{Proposition}
  \providecommand{\theoremname}{Theorem}
\providecommand{\theoremname}{Theorem}

\begin{document}

\begin{frontmatter}{}

\title{A numerical framework for computing steady states of size-structured
population models and their stability}

\author{Inom Mirzaev}

\ead{mirzaev@colorado.edu}

\author{David M. Bortz}

\ead{dmbortz@colorado.edu}

\address{Department of Applied Mathematics, University of Colorado, Boulder,
CO, 80309-0526, United States}
\begin{abstract}
Structured population models are a class of general evolution equations
which are widely used in the study of biological systems. Many theoretical
methods are available for establishing existence and stability of
steady states of general evolution equations. However, except for
very special cases, finding an analytical form of stationary solutions
for evolution equations is a challenging task. In the present paper,
we develop a numerical framework for computing approximations to stationary
solutions of general evolution equations, which can \emph{also} be
used to produce existence and stability regions for steady states.
In particular, we use the Trotter-Kato Theorem to approximate the
infinitesimal generator of an evolution equation on a finite dimensional
space, which in turn reduces the evolution equation into a system
of ordinary differential equations. Consequently, we approximate and
study the asymptotic behavior of stationary solutions. We illustrate
the convergence of our numerical framework by applying it to a linear
Sinko-Streifer structured population model for which the exact form
of the steady state is known. To further illustrate the utility of
our approach, we apply our framework to nonlinear population balance
equation, which is an extension of well-known Smoluchowksi coagulation-fragmentation
model to biological populations. We also demonstrate that our numerical
framework can be used to gain insight about the theoretical stability
of the stationary solutions of the evolution equations. Furthermore,
the open source Python program that we have developed for our numerical
simulations is freely available from our Github repository \emph{(github.com/MathBioCU}). \end{abstract}

\end{frontmatter}{}

\section{Introduction}

Many natural phenomena can be formulated as the differential law of
the development (evolution) in time of a physical system. The resulting
differential equation combined with boundary conditions affecting
the system are called \emph{evolution equations}. Evolution equations
are a popular framework for studying the dynamics of biological populations.
For example, they have proven useful in understanding the dynamics
of biological invasions \citep{schreiber2011invasion}, bacterial
flocculation in activated sludge tanks \citep{biggs2002modelling},
and the spread of parasites and diseases \citep{gourley2008spatiotemporal}.
Since many biological populations converge to a time-independent state,
many researchers have used theoretical tools to investigate long-term
behavior of these models. Analytical and fixed point methods have
been used to show the existence of stationary solutions to size-structured
population models \citep{mirzaev2015stability,farkas2012steadystates}
and semigroup theoretic methods have been used to investigate the
stability of these stationary solutions \citep{farkas2007stability,mirzaev2015criteria,banasiak2011blowupof}.
For many models in the literature, the principle of linearized stability
\citep{webb1985theoryof,kato1995aprinciple} can be used to show that
the spectral properties of the infinitesimal generator (IG) of the
linearized semigroup determines the stability or instability of a
stationary solution. Moreover, using compactness arguments, spectral
properties of the infinitesimal generator can be determined from the
point spectrum of the IG, which in turn can be written as the roots
of a characteristic equation.

Despite this theoretical development, the derived existence and stability
conditions are oftentimes rather complex, and accordingly the biological
interpretation of these conditions can be challenging. To overcome
this difficulty several numerical methods for stability analysis of
structured population models have been developed \citep{breda2004computing,engelborghs2002numerical,kirkilionis2001numerical,deroos2010numerical}.
For instance, Diekmann et al. \citep{kirkilionis2001numerical,diekmann2003steadystate,deroos2010numerical}
present a numerical method for equilibrium and stability analysis
of physiologically structured population models (PSPM) or life history
models, where individuals are characterized by a (finite) set of physiological
characteristics (traits such as age, size, sex, energy reserves).
In this method a PSPM is first written as a system of integral equations
coupled with each other via interaction (or feedback) variables. Consequently,
equilibria and stability boundary of the resulting integral equations
are numerically approximated using curve tracing methods combined
with numerical integration of the ODE. Later, de Roos \citep{deroos2008demographic}
presented the modification of the curve tracing approach to compute
the demographic characteristics (such as population growth rate, reproductive
value, etc) of a linear PSPM. For a fast and accurate software for
theoretical analysis of PSPMs we refer interested reader to a software
package by de Roos \citep{deroos2014pspmanalysis}. An alternative
method for stability analysis of physiologically structured population
models, developed by Breda and coworkers \citep{breda2013anumerical,breda2005pseudospectral,breda2006solution},
uses a pseudospectral approach to compute eigenvalues of a discretized
infinitesimal generator. This method (also known as infinitesimal
generator (IG) approach) has been employed to produce bifurcation
diagrams and stability regions of many different linear evolution
equations arising in structured population modeling \citep{breda2005pseudospectral,breda2006pseudospectral,breda2009numerical}.
Unfortunately, not all structured population models fit into the framework
of PSPMs and thus there is a need for a more general numerical framework
for asymptotic analysis of structured population models.

In this paper we develop a numerical framework to guide theoretical
analysis of structured population models. We demonstrate that our
methodology can be used for numerical computation and stability analysis
of positive stationary solutions of both linear and nonlinear size-structured
population models. Moreover, we illustrate the utility of our framework
to produce existence and stability regions for steady states of size-structured
population models. We also provide an open source Python program used
for the numerical simulations in our Github repository \citep{mirzaev2015steadystate}.
Although, the examples presented in this paper are size-structured
population models, in Section \ref{sec:Theoretical-Framework}, we
show that the framework is applicable to more general evolution equations. 

The main idea behind the numerical framework is first to write a structured
population model in the form of an evolution equation and then use
well-known Trotter-Kato Theorem \citep{trotter1958approximation,kato1959remarks}
to approximate the infinitesimal generator of the evolution equation
on a finite dimensional space. This in turn allows one to approximate
solutions (or spectrum) of the evolution equation with the solution
(or spectrum) of system of differential equations. Consequently, we
approximate the stationary solutions of an actual model with stationary
solutions of the approximate infinitesimal generator on a finite dimensional
space. Local stability of the approximate steady states are then computed
from the spectrum of the Jacobian of ODE system evaluated at this
steady states. Our method is similar to the IG approach in \citep{breda2005pseudospectral,breda2006pseudospectral,breda2009numerical},
in a sense that we also approximate infinitesimal generator and analyze
the spectrum of the resulting operator to produce existence and stability
regions. However, in contrast to IG approach, our framework also computes
actual steady states and is easily applied to nonlinear evolution
equations arising in structured population dynamics. 

The rest of the paper is structured as follows. We describe the theoretical
development of our framework for general evolution equations in Section
\ref{sub:Approximation-scheme}. Note that readers with more biological
background can skip Section \ref{sub:Approximation-scheme} and directly
jump into the application of the framework in Section \ref{sub:Numerical-convergence-results}.
In Section \ref{sub:Numerical-convergence-results}, we illustrate
the convergence of the approximation method by applying it to linear
Sinko-Streifer model, for which the exact form of the stationary solutions
is known. To further illustrate the utility of our approach, in Section
\ref{sec:Application-to-nonlinear}, we apply our framework to a nonlinear
size-structure population model (also known as population balance
equations in the engineering literature) described in \citep{banasiak2009coagulation,bortz2008klebsiella}.
Moreover, in Section \ref{sec:Establishing-local-stability}, we show
that local stability conditions for a stationary solution can be derived
from the spectral properties of the approximate infinitesimal generator.
Finally, we conclude with some remarks and a summary of our work in
Section 5.

\section{\label{sec:Theoretical-Framework}Numerical Framework}

In this section, we demonstrate our numerical methodology for general
evolution equations. We first present the numerical scheme used to
approximate the infinitesimal generator of a semigroup. Subsequently,
in Section \ref{sub:Numerical-convergence-results}, we illustrate
the convergence of our approach by applying it to linear Sinko-Streifer
equations, for which exact stationary solutions are known.

\subsection{\label{sub:Approximation-scheme}Infinitesimal generator approximation}

Given a Banach space $\mathcal{X}$, consider an abstract evolution
equation,
\begin{equation}
u_{t}=\mathcal{F}(u),\qquad u(0,\bullet)=u_{0}\in\mathcal{X}\,,\label{eq:abstract evolution equation}
\end{equation}
where $\mathcal{F}\,:\,\mathcal{D}(\mathcal{F})\subseteq\mathcal{X}\to\mathcal{X}$
is some operator defined on its domain $\mathcal{D}(\mathcal{F})$
and $u_{0}$ is an initial condition at time $t=0$. Note that any
boundary condition belonging to a given partial differential equation
can be included in the domain $\mathcal{D}(\mathcal{F})$. The solution
to (\ref{eq:abstract evolution equation}) can be related to the initial
state $u_{0}$ by some transition operator $T(t)$ so that 
\[
u(t,\,x)=T(t)u_{0}(x)\,.
\]
The transition operator $T(t)$ is said to be strongly continuous
semigroup (or simply $C_{0}$ -semigroup) if satisfies the following
three conditions:
\begin{enumerate}
\item $T(s)T(t)=T(s+t)$ for all $s,\,t\ge0$
\item $T(0)=I$, the identity operator on $\mathcal{X}$
\item For each fixed $u_{0}\in\mathcal{X}$,
\[
\lim_{t\to0^{+}}\left\Vert T(t)u_{0}-u_{0}\right\Vert =0\,.
\]

\end{enumerate}
Moreover, showing that the operator $\mathcal{F}$ generates a $C_{0}$-semigroup
is equivalent to establishing well-posedness of the abstract evolution
equation given in (\ref{eq:abstract evolution equation}). In general,
finding the explicit form of the transition operator is challenging.
Hence, approximation methods are used to study a more complicated
evolution equation and the semigroups they generate. One of the famous
approximation theorems is due to Trotter \citep{trotter1958approximation}
and Kato \citep{kato1959remarks} (see \citep{kato1976perturbation}
for the classical literature on the approximation of generators of
semigroups). The goal is to construct approximating generators $\mathcal{F}_{n}$
on the approximate spaces $\mathcal{X}_{n}$ such that $C_{0}$-semigroups
$T_{n}(\cdot)$ generated by $\mathcal{F}_{n}$ approximate the $C_{0}$-semigroup
$T(t)$ generated by $\mathcal{F}$. Although there are multiple ways
to approximate the infinitesimal generator $\mathcal{F}$, for our
purposes we use the approximation scheme based on Galerkin-type projection
of the state space $\mathcal{X}$ \citep{banks1989transformation,ito1998thetrotterkato,ackleh1997modeling}.
For the convenience of readers, we will summarize the approximation
scheme here. 

Let $\mathcal{X}_{n}$, $n=1,2,\dots$ be a sequence of subspaces
of $\mathcal{X}$ with $\dim(\mathcal{X}_{n})=n$ and define projections
$\pi_{n}\,:\,\mathcal{X}\to\mathcal{X}_{n}$ and canonical injections
$\iota_{n}\,:\,\mathcal{X}_{n}\to\mathcal{X}$. For each subspace
$\mathcal{X}_{n}$ we choose basis elements $\left\{ \beta_{i}^{n}\right\} _{i=1}^{n}$
such that each element $v$ of subspace $\mathcal{X}_{n}$ can be
written in the form $v=\sum_{i=1}^{n}\alpha_{i}\beta_{i}^{n}$. Moreover,
for each subspace $\mathcal{X}_{n}$ we define the bijective mappings
$p_{n}\,:\,\mathcal{X}_{n}\to\mathbb{R}^{n}$ by 
\[
p_{n}v=(\alpha_{1},\cdots,\alpha_{n})^{T}
\]
 and the norm on $\mathbb{R}^{n}$ by 
\[
\left\Vert v\right\Vert _{\mathbb{R}^{n}}=\left\Vert p_{n}^{-1}v\right\Vert _{\mathcal{X}}\,.
\]
Consequently, we define bounded linear operators $P_{n}\,:\,\mathcal{X}\to\mathbb{R}^{n}$
and $E_{n}\,:\,\mathbb{R}^{n}\to\mathcal{X}$ by 
\begin{equation}
P_{n}v=p_{n}\pi_{n}v,\quad v\in\mathcal{X}\label{eq: space to reals}
\end{equation}
and 
\begin{equation}
E_{n}z=\iota_{n}p_{n}^{-1}z,\quad z\in\mathbb{R}^{n}\,,\label{eq:reals to space}
\end{equation}
respectively. Finally, we define approximate operators $\mathcal{F}_{n}$
on $\mathbb{R}^{n}$ by
\begin{equation}
\mathcal{F}_{n}(z)=P_{n}\mathcal{F}\left(E_{n}z\right)\label{eq:approximate LHS}
\end{equation}
for all $z\in\mathbb{R}^{n}$. 

Accordingly, the Trotter-Kato Theorem states that the semigroup generated
by the discrete operator $\mathcal{F}_{n}$ converges to the semigroup
generated by the infinitesimal generator $\mathcal{F}$. For the convenience
of the reader, we state the theorem here and refer readers to \citep{ito1998thetrotterkato}
for a proof.
\begin{thm*}
(Trotter-Kato) Let $\left(T(t)\right)_{t\ge0}$ and $\left(T_{n}(t)\right)_{t\ge0}$,
$n\in\mathbb{N}$, be strongly continuous semigroups on $\mathcal{X}$
and $\mathbb{R}^{n}$ with generators $\mathcal{F}$ and $\mathcal{F}_{n}$,
respectively. Furthermore, assume that they satisfy the estimate
\[
\left\Vert T(t)\right\Vert _{\mathcal{X}},\,\left\Vert T_{n}(t)\right\Vert _{\mathbb{R}^{n}}\le Me^{wt}\qquad\text{for all }t\ge0,\,n\in\mathbb{N}\,,
\]
for some constants $M\ge1,\,w\in\mathbb{R}$. Then the following are
equivalent:
\begin{enumerate}
\item There exists a $\lambda_{0}\in\rho(\mathcal{F})\cap\bigcap\limits _{i=1}^{n}\rho(\mathcal{F}_{i})$
such that for all $v\in\mathcal{X}$
\[
\left\Vert E_{n}\left(\lambda_{0}I_{n}-\mathcal{F}_{n}\right)^{-1}P_{n}v-\left(\lambda_{0}I-\mathcal{F}\right)^{-1}v\right\Vert _{\mathcal{X}}\to0\quad\text{as }n\to\infty\,.
\]

\item For all $v\in\mathcal{X}$ and $t\ge0$,
\[
\left\Vert E_{n}T_{n}(t)P_{n}v-T(t)v\right\Vert _{\mathcal{X}}\to0
\]
as $n\to\infty$, uniformly on compact $t$ intervals.
\end{enumerate}
\end{thm*}
In general, one establishes the first statement for a Trotter-Kato
approximation and then uses the second statement to approximate an
abstract evolution equation on a finite dimensional space. In their
paper, Ito and Kappel \citep{ito1998thetrotterkato} present the standard
ways to establish the first statement of the theorem (see also \citep{banks1989transformation,ackleh1997modeling,ackleh1997parameter}).
Therefore, here we assume that for a particular problem the first
statement in the theorem has already been established and thus the
evolution equation in (\ref{eq:abstract evolution equation}) can
be approximated by the following system of ODEs, 
\begin{equation}
u_{n}'(t)=\mathcal{F}_{n}\left(u_{n}(t)\right),\quad u_{n}(0)=P_{n}u(0,\,\bullet)\,.\label{eq:approximate IVP}
\end{equation}
Consequently, the solution of the IVP is mapped onto the infinite
dimensional Banach space $\mathcal{X}$ and one has the following
convergence 
\begin{equation}
\lim_{n\to\infty}\left\Vert E_{n}u_{n}(t)-u\right\Vert _{\mathcal{X}}=0\label{eq:convergence of IVP solutions}
\end{equation}
for $t$ in compact intervals. 

In general, finding explicit stationary solutions of abstract evolution
equations is a challenging task. Conversely, many efficient root finding
methods have been developed for finding steady states of a system
of ODEs. For large-scale nonlinear systems, many efficient methods
have been developed as well. Hence, we propose a numerical framework
that utilizes those efficient root finding methods to approximate
steady state solutions of general evolution equations. The idea is
to use an efficient and accurate root finding method to approximate
a stationary solution of the evolution equation (\ref{eq:abstract evolution equation})
with the stationary solutions of the IVP in (\ref{eq:approximate IVP}).
Thus, as a consequence of the Trotter-Kato Theorem, the steady states
of (\ref{eq:approximate IVP}) converge to the steady states of (\ref{eq:abstract evolution equation})
as $n\to\infty$.

\subsection{\label{sub:Numerical-convergence-results}Numerical convergence results}

To verify convergence of the proposed approximation scheme, we apply
the framework to the linear Sinko-Streifer model \citep{sinko1967anew}
for which an exact form of the stationary solution is available. The
model describes the dynamics of single species populations and takes
into account the physiological characteristics of animals of different
sizes (and/or ages) . The mathematical model reads as
\begin{equation}
u_{t}=\mathcal{G}(u)=-(gu)_{x}-\mu u,\quad t\ge0,\quad0\le x\le\overline{x}<\infty\label{eq:sinko-streifer}
\end{equation}
with a McKendrick-von Foerster type renewal boundary condition at
$x=0$
\[
g(0)u(t,\,0)=\int_{0}^{\overline{x}}q(y)u(t,\,y)\,dy
\]
and initial condition
\[
u(0,\,x)=u_{0}(x)\,.
\]
The variable $u(t,\,x)$ denotes the population density at time $t$
with size class $x$. The population is assumed to have a minimum
and a maximum size $0$ and $\overline{x}<\infty$, respectively.
The function $g(x)$ represents the average growth rate of the size
class $x$ and the coefficient $\mu(\bullet)$ represents a size-dependent
removal rate due to death or predation. The renewal function $q(\bullet)$
represents the number of new individuals entering the population due
to birth. 

Setting the right side of the equation (\ref{eq:sinko-streifer})
to zero and integrating over the size on $(0,\,x)$ yields the exact
stationary solution
\begin{equation}
u_{*}(x)=\frac{1}{g(x)}\exp\left(-\int_{0}^{x}\frac{\mu(s)}{g(s)}\,ds\right)\int_{0}^{\overline{x}}q(y)u_{*}(y)\,dy\,.\label{eq:exact stationary solution}
\end{equation}
Multiplying both sides of (\ref{eq:exact stationary solution}) by
$q(x)$ integrating over the size on $(0,\,\overline{x})$, we obtain
a necessary condition for existence of a stationary solution,
\begin{equation}
1=\int_{0}^{\overline{x}}\frac{q(x)}{g(x)}\exp\left(-\int_{0}^{x}\frac{\mu(s)}{g(s)}\,ds\right)\,dx\,.\label{eq:necessary condition}
\end{equation}
We also note that if $u_{*}$ is a stationary solution satisfying
(\ref{eq:exact stationary solution}), then any multiple of $u_{*}$
is also stationary solution of (\ref{eq:sinko-streifer}). The convergence
of the approximation scheme presented in Section \ref{sub:Approximation-scheme}
for Sinko-Streifer models has already been established in \citep{banks1989transformation}.
Using the basis functions for $n$-dimensional subspace $\mathcal{X}_{n}$
of the state space $\mathcal{X}=L^{1}(0,\,\overline{x})$ are defined
as
\[
\beta_{i}^{n}(x)=\left\{ \begin{array}{cc}
1;\, & x_{i-1}^{n}<x\le x_{i}^{n}\,;\,i=1,\ldots,n\\
0;\, & \mbox{otherwise}
\end{array}\right.
\]
for positive integer $n$ with $\{x_{i}^{n}\}_{i=0}^{n}$ a uniform
partition of $[0,\overline{x}]$, and $\Delta x=x_{j}^{n}-x_{j-1}^{n}$
for all $j$. The functions $\beta^{n}$ form an orthogonal basis
for the approximate solution space 
\[
\mathcal{X}_{n}=\left\{ h\in\mathcal{X}\;|\;h=\sum_{i=1}^{n}\alpha_{i}\beta_{i}^{n},\;\alpha_{i}\in\mathbb{R}\right\} ,
\]
and accordingly, we define the orthogonal projections $\pi_{n}:\,\mathcal{X}\to\mathcal{X}_{n}$
\[
\pi_{n}h(x)=\sum_{j=1}^{n}\,\alpha_{j}\beta_{j}^{n}(x),\qquad\text{where}\:\alpha_{j}=\frac{1}{\Delta x}\,\int_{x_{j-1}^{n}}^{x_{j}^{n}}\,h(x)\,dx.
\]
Moreover, since the evolution equation defined in (\ref{eq:sinko-streifer})
is a linear partial differential equation, the approximate operator
$\mathcal{G}_{n}$ on $\mathbb{R}^{n}$ is given by the following
$n\times n$ matrix{\footnotesize{}
\begin{equation}
\mathcal{G}_{n}=\begin{pmatrix}-\frac{1}{\Delta x}g(x_{1}^{n})-\mu(x_{1}^{n})+q(x_{1}^{n}) & q(x_{2}^{n}) & \cdots & q(x_{n-1}^{n}) & q(x_{n}^{n})\\
\frac{1}{\Delta x}g(x_{1}^{n}) & -\frac{1}{\Delta x}g(x_{2}^{n})-\mu(x_{2}^{n}) & 0 & \cdots & 0\\
0 & \frac{1}{\Delta x}g(x_{2}^{n}) & \ddots & \ddots & \vdots\\
\vdots & \ddots & \frac{1}{\Delta x}g(x_{n-2}^{n}) & -\frac{1}{\Delta x}g(x_{n-1}^{n})-\mu(x_{n-1}^{n}) & 0\\
0 & \cdots & 0 & \frac{1}{\Delta x}g(x_{n-1}^{n}) & -\frac{1}{\Delta x}g(x_{n}^{n})-\mu(x_{n}^{n})
\end{pmatrix}\,.\label{eq:approximate G}
\end{equation}
}At this point, one can use numerical techniques to calculate zeros
of the linear system
\begin{equation}
\mathcal{G}_{n}u_{n}=0\,.\label{eq:steady state equation}
\end{equation}
For the purpose of illustration, we set the model rates to 
\begin{equation}
q(x)=a(x+1),\quad g(x)=b(x+1),\quad\mu(x)=c\,.\label{eq:sinko region}
\end{equation}
For the purpose of illustration, we arbitrarily choose $a=1/\ln2$
and $b=c=1$ as these values satisfy the necessary condition (\ref{eq:necessary condition})
for the existence of the steady states of the Sinko-Streifer model.
Since the approximate operator $\mathcal{G}_{n}$ is an $n\times n$
matrix, we can compute the nullspace of $\mathcal{G}_{n}$ using standard
tools. The results of the numerical simulations are depicted in Figure
\ref{fig:Results-of-the}. Figure \ref{fig:convergence} illustrates
that as the dimension of the approximate space $\mathcal{X}_{n}$
increases the absolute error between the exact stationary solution
and the approximate stationary solution decreases. Moreover, the numerical
algorithm has a linear convergence rate. This is due to the fact that
we chose zeroth order functions as basis functions for approximate
subspaces. In general, if one desires a higher order convergence for
Galerkin-type approximations, choosing higher order basis functions
gives higher convergence rate \citep{kappel1981splineapproximations}.
Furthermore, Figure \ref{fig:comparison} indicates that even for
$n=100$ the fit between approximate and actual stationary solution
is satisfactory (the infinity norm of the error is $0.14$).

\begin{figure}
\centering{}\subfloat[\label{fig:convergence}]{\protect\includegraphics[width=0.33\columnwidth]{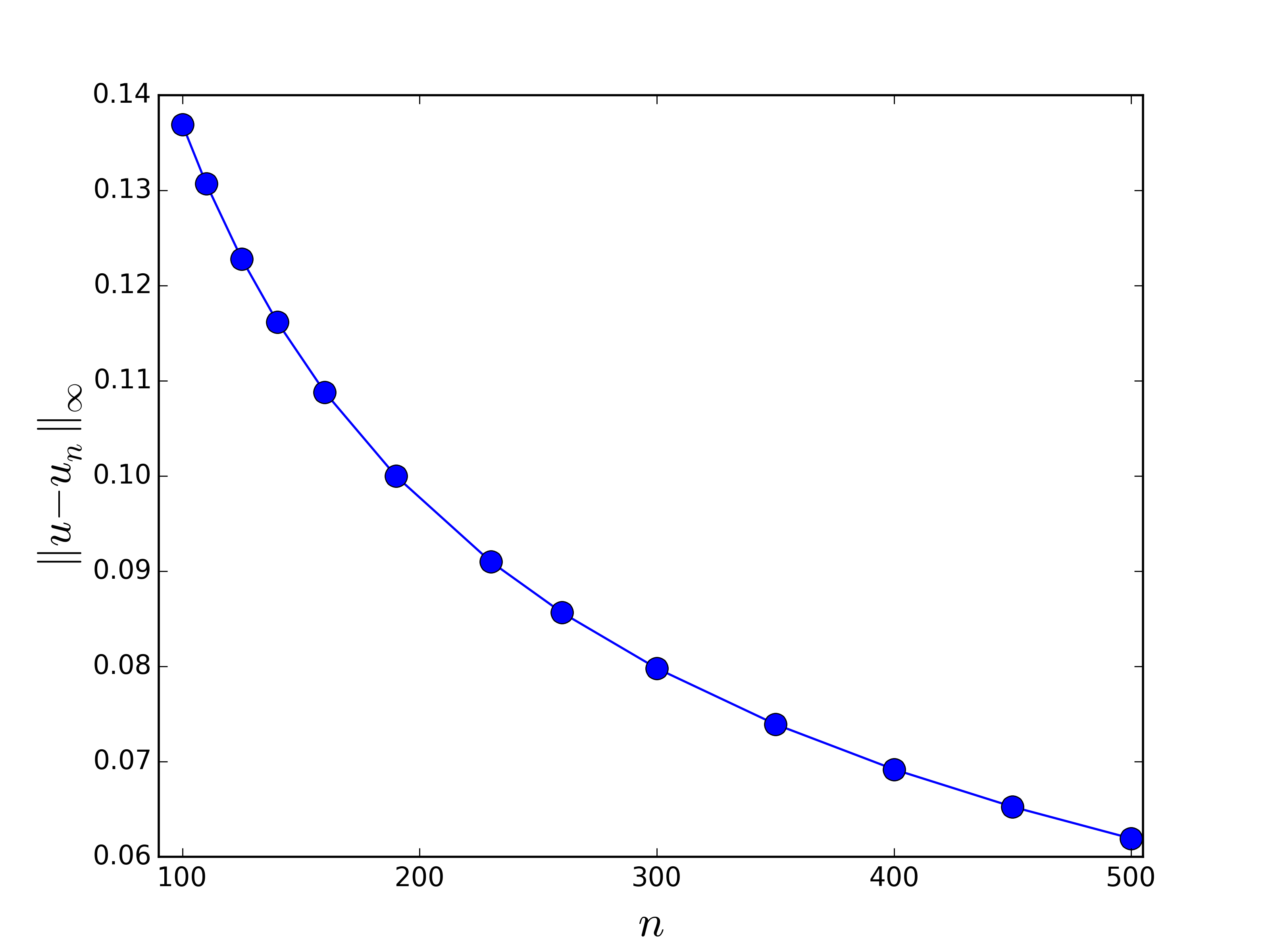}}~\subfloat[\label{fig:comparison}]{\protect\includegraphics[width=0.33\columnwidth]{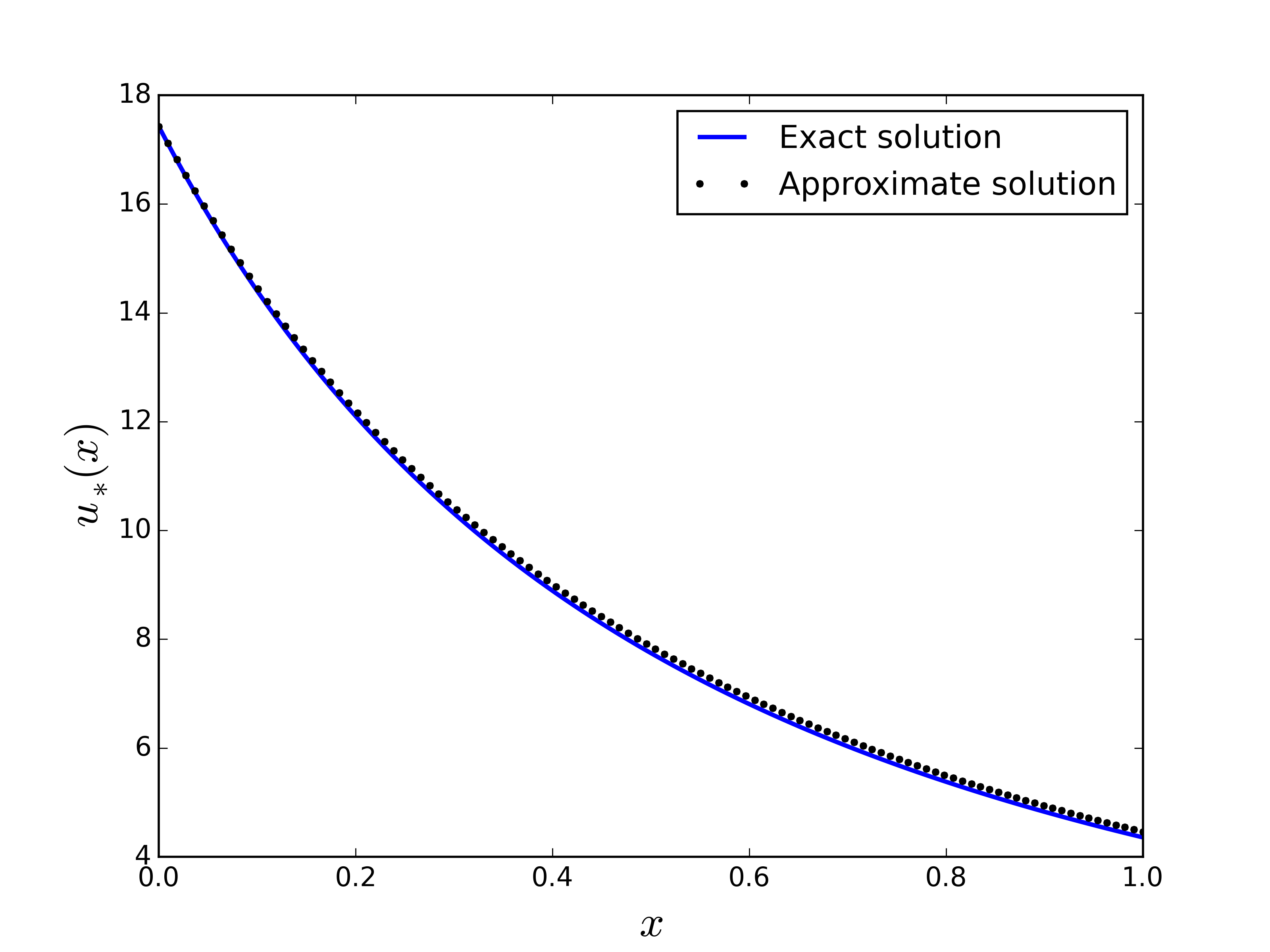}}~\subfloat[\label{fig:sinko-Existence-region}]{\protect\raggedright{}\protect\includegraphics[width=0.33\paperwidth]{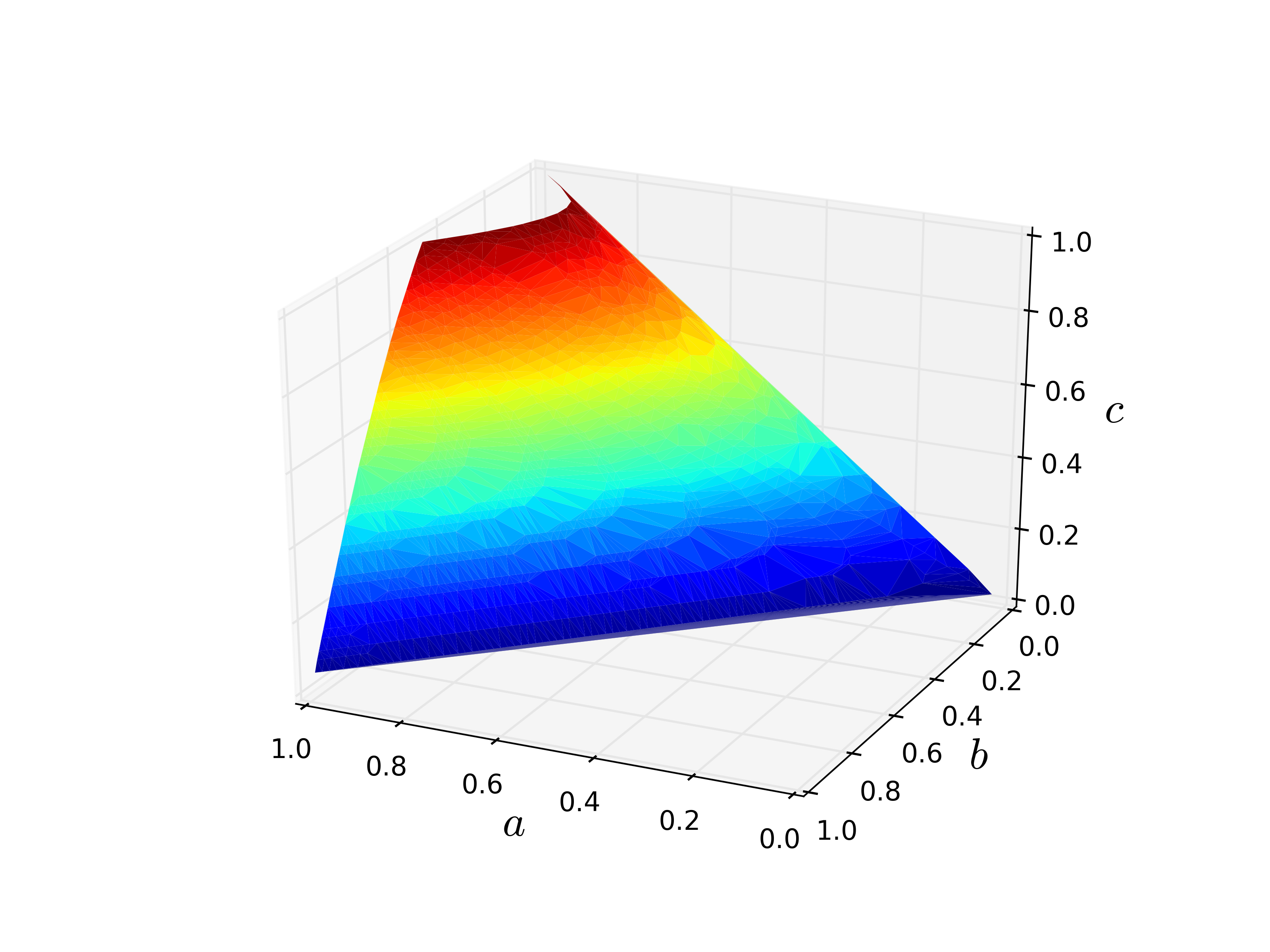}\protect}\protect\caption{\label{fig:Results-of-the}Results of the numerical simulations. a)
Absolute error between exact stationary solution and approximate stationary
solution decreases as the dimension of approximate subspaces $\mathcal{X}_{n}$
increase. b) Comparison of exact stationary solution with approximate
stationary solution for $n=100$. c) Existence and stability surface
for the steady states of the Sinko-Streifer model.}
\end{figure}

To further illustrate the utility of our approach, we used the numerical
scheme described in this section to generate existence and stability
regions for Sinko-Streifer model. Particularly, the interval $(a,b,\,c)\in[0,\,1]\times[0,\,1]\times[0,\,1]$
is discretized with $\Delta a=\Delta b=\Delta c=0.01$. Consequently,
we checked for the existence of a positive steady state at each of
these discrete points, the resulting existence region is depicted
in Figure \ref{fig:sinko-Existence-region} forming a nontrivial three-dimensional
surface. Moreover, the existence and stability regions of the Sinko-Streifer
model coincide for the chosen model rates in (\ref{eq:sinko region}).

\section{\label{sec:Application-to-nonlinear}Application to nonlinear population
balance equation}

In aerosol physics and environmental sciences, studying the \emph{flocculation}
of particles is widespread. The process of flocculation involves disperse
particles in suspension combining into aggregates (i.e., a floc) and
separating. The mathematical model used to study flocculation process
is the well-known population balance equation (PBE) which describes
the time-evolution of the particle size number density. The equations
for the flocculation model track the time-evolution of the particle
size number density $u(t,\,x)$ and can be written as 
\begin{eqnarray}
\partial_{t}u & = & \mathcal{F}(u)\label{eq: agg and growth model}
\end{eqnarray}
where
\[
\mathcal{F}(u):=\mathcal{G}(u)+\mathcal{A}(u)+\mathcal{B}(u),
\]
$\mathcal{G}$ denotes growth
\begin{equation}
\mathcal{G}(u):=-\partial_{x}(gu)-\mu(x)u(t,\,x)\,,\label{eq:Growth}
\end{equation}
$\mathcal{A}$ denotes aggregation
\begin{alignat}{1}
\mathcal{A}(u) & :=\frac{1}{2}\int_{0}^{x}k_{a}(x-y,\,y)u(t,\,x-y)u(t,\,y)\,dy\nonumber \\
 & \quad-u(t,\,x)\int_{0}^{\overline{x}-x}k_{a}(x,\,y)u(t,\,y)\,dy\,,\label{eq:Aggregation}
\end{alignat}
and $\mathcal{B}$ denotes breakage 
\begin{equation}
\mathcal{B}(u):=\int_{x}^{\overline{x}}\Gamma(x;\,y)k_{f}(y)u(t,\,y)\,dy-\frac{1}{2}k_{f}(x)u(t,\,x)\,.\label{eq:Breakage}
\end{equation}
The boundary condition is traditionally defined at the smallest size
$0$ and the initial condition is defined at $t=0$

\[
g(0)u(t,\,0)=\int_{0}^{\overline{x}}q(x)u(t,\,x)dx,\quad u(0,\,x)=u_{0}(x)\in L^{1}(0,\,\overline{x})\,,
\]
where the renewal rate $q(x)$ represents the number of new flocs
entering the population. A floc is assumed to have a maximum size
$\overline{x}<\infty$. The function $g(x)$ represents the average
growth rate of the flocs of size $x$ due to proliferation, and the
coefficient $\mu(x)$ represents a size-dependent removal rate due
to gravitational sedimentation and death. The function $k_{a}(x,\,y)$
is the aggregation kernel, which describes the rate with which the
flocs of size $x$ and $y$ agglomerate to form a floc of size $x+y$.
The fragmentation kernel $k_{f}(x)$ calculates the rate with which
a floc of size $x$ fragments. The integrable function $\Gamma(x;y)$
represents the post-fragmentation probability density of daughter
flocs for the fragmentation of the parent flocs of size $y$. In other
words, all the fractions of daughter flocs formed upon the fragmentation
of a parent floc sum to unity,
\begin{equation}
\int_{0}^{y}\Gamma(x;\,y)\,dx=1\text{ for all }y\in(0,\,\overline{x}].\label{eq:number conservation}
\end{equation}

The population balance equation, presented in (\ref{eq: agg and growth model}),
is a generalization of many mathematical models appearing in the size-structured
population modeling literature and has been widely used, e.g., to
model the formation of clouds and smog in meteorology \citep{pruppacher2012microphysics},
the kinetics of polymerization in biochemistry \citep{ziff1980kinetics},
the clustering of planets, stars and galaxies in astrophysics \citep{makino1998onthe},
and even schooling of fish in marine sciences \citep{niwa1998schoolsize}.
For example, when the fragmentation kernel is omitted, $k_{f}\equiv0$,
the flocculation model reduces to algal aggregation model used to
describe evolution of a phytoplankton community \citep{ackleh1997modeling}.
When the removal and renewal rates are set to zero, the flocculation
model simplifies to a model used to describe the proliferation of
\emph{Klebsiella pneumoniae }in a bloodstream \citep{bortz2008klebsiella}.
Furthermore, the flocculation model, with only growth and fragmentation
terms, was used to investigate the elongation of prion polymers in
infected cells \citep{calvez2012selfsimilarity,doumic-jauffret2009eigenelements,calvez2010priondynamics}. 

The equation (\ref{eq: agg and growth model}) has also been the focus
of considerable mathematical analysis. Well-posedness of the general
flocculation model was first established by Ackleh and Fitzpatrick
\citep{ackleh1997parameter,ackleh1997modeling} in an $L^{2}$-space
setting and later by Banasiak and Lamb \citep{banasiak2009coagulation}
in an $L^{1}$-space setting. Moreover, asymptotic behavior of the
equation (\ref{eq: agg and growth model}) has been a challenging
task because of the nonlinearity introduced by the aggregation terms.
Nevertheless, under suitable conditions on the kernels, the existence
of a positive steady state has been established for the pure aggregation
and fragmentation case \citep{laurencot2005steadystates}. For a review
of further mathematical results, we refer readers to review articles
by Menon and Pego \citep{menon2006dynamical}, and Wattis \citep{wattis2006anintroduction}
and the book by Ramkrishna \citep{ramkrishna2000population}. Lastly,
although the population balance equation has received substantial
theoretical work, the derivation of analytical solutions for many
realistic aggregation kernels has proven elusive. Towards this end,
many discretization schemes for numerical simulations of the PBEs
have been proposed. For instance, to approximate steady state solutions
of PBEs, numerical schemes based on the least squares spectral method
\citep{dorao2006application,dorao2006aleast,dorao2007leastsquares}
and the finite element method \citep{nicmanis1998finiteelement,nicmanis2002errorestimation,hounslow1990adiscretized}
have been developed. For the further review of approximation methods
we refer interested readers to the review by Bortz \citep{bortz2015chapter}.

\subsection{\label{sub:Numerical-implementation-and}Numerical implementation
and results}

For the numerical implementation we adopt the scheme developed in
Section \ref{sub:Approximation-scheme}. Therefore, the approximate
formulation of (\ref{eq: agg and growth model}) becomes the following
system of $n$ nonlinear ODEs for $u_{n}=(\alpha_{1},\cdots,\alpha_{n})^{T}\in\mathbb{R}^{n}$
:

\begin{alignat}{1}
\dot{u}_{n} & =\mathcal{F}_{n}(u_{n})=\mathcal{G}_{n}u_{n}+P_{n}\mathcal{A}(E_{n}u_{n})+P_{n}\mathcal{B}(E_{n}u_{n}),\label{eq:numerical approximation}\\
u_{n}(0,x) & =P_{n}u_{0}(x)\,,\label{eq:initial condition projection}
\end{alignat}
where the matrix $\mathcal{G}_{n}$ is defined as in Section \ref{sub:Numerical-convergence-results},
\[
P_{n}\mathcal{A}(E_{n}u_{n})=\begin{pmatrix}-\alpha_{1}\sum_{j=1}^{n-1}k_{a}(x_{1}^{n},\,x_{j}^{n})\alpha_{j}\Delta x\\
\frac{1}{2}k_{a}(x_{1}^{n},\,x_{1}^{n})\alpha_{1}\alpha_{1}\Delta x-\alpha_{2}\sum_{j=1}^{n-2}k_{a}(x_{2}^{n},\,x_{j}^{n})\alpha_{j}\Delta x\\
\vdots\\
\frac{1}{2}\sum_{j=1}^{n-2}k_{a}(x_{j}^{n},\,x_{n-1-j}^{n})\alpha_{j}\alpha_{n-1-j}\Delta x-\alpha_{n-1}k_{a}(x_{n-1}^{n},\,x_{1}^{n})\alpha_{1}\Delta x\\
\frac{1}{2}\sum_{j=1}^{n-1}k_{a}(x_{j}^{n},\,x_{n-j}^{n})\alpha_{j}\alpha_{n-j}\Delta x
\end{pmatrix}
\]
and 
\[
P_{n}\mathcal{B}(E_{n}u_{n})=\begin{pmatrix}\sum_{j=2}^{n}\Gamma(x_{1}^{n};\,x_{j}^{n})k_{f}(x_{j}^{n})\alpha_{j}\Delta x-\frac{1}{2}k_{f}(x_{1}^{n})\alpha_{1}\\
\sum_{j=3}^{n}\Gamma(x_{2}^{n};\,x_{j}^{n})k_{f}(x_{j}^{n})\alpha_{j}\Delta x-\frac{1}{2}k_{f}(x_{2}^{n})\alpha_{2}\\
\vdots\\
\Gamma(x_{n-1}^{n};\,x_{n}^{n})k_{f}(x_{n}^{n})\alpha_{n}\Delta x-\frac{1}{2}k_{f}(x_{n-1}^{n})\alpha_{n-1}\\
-\frac{1}{2}k_{f}(x_{n}^{n})\alpha_{n}
\end{pmatrix}\,.
\]
The convergence of the approximate scheme (\ref{eq:numerical approximation})-(\ref{eq:initial condition projection})
has been established in \citep{ackleh1997parameter}. Therefore, the
stationary solutions of the microbial flocculation model (\ref{eq: agg and growth model})
can be systematically approximated by the stationary solutions of
the system of nonlinear ODEs given in (\ref{eq:numerical approximation}).
We used Powell's hybrid root finding method \citep{powell1970ahybrid}
as implemented in Python 2.7.10\texttt{ }\footnote{\texttt{scipy.optimize.fsolve}}
to find zeros of the steady state equation \citep{mirzaev2015steadystate}.
For faster convergence rate, we provided the solver with the exact
Jacobian of $\mathcal{F}_{n}(u_{n})$ ( see Section \ref{sec:Establishing-local-stability},
Eqn (\ref{eq:jacobian of F}) for the formulation of the Jacobian.
For the purpose of illustration, for a post-fragmentation density
function we chose the well-known Beta distribution\footnote{ Although normal and log-normal distributions are mostly used in the
literature, Byrne et al. \citep{byrne2011postfragmentation} have
provided evidence that the Beta density function describes the fragmentation
of small bacterial flocs. } with $\alpha=\beta=2$,
\[
\Gamma(x,\,y)=\mathbbm{1}_{[0,\,y]}(x)\frac{6x(y-x)}{y^{3}}\,,
\]
where $\mathbbm{1}_{I}$ is the indicator function on the interval
$I$. The aggregation kernel was chosen to describe flow within laminar
shear field (i.e., \emph{orthokinetic} aggregation)
\[
k_{a}(x,\,y)=\left(x^{1/3}+y^{1/3}\right)^{3}
\]
Other model rates were chosen arbitrarily as 
\[
q(x)=a(x+1),\quad g(x)=b(x+1)\quad\mu(x)=cx\quad k_{f}(x)=x\,,
\]
where $a,\,b$ and $c$ are some positive real numbers. 

The main advantage of this approximation scheme (\ref{eq:numerical approximation})-(\ref{eq:initial condition projection})
is that it can be initialized very fast using Toeplitz matrices \citep{matveev2015afast}.
Fast initialization of the discretization scheme allows one to check
the existence of the steady states at many discrete points efficiently.
This in turn allows for the generation of the existence and stability
regions of the steady states of the PBE in (\ref{eq: agg and growth model}).
To illustrate the existence regions of the steady states of the PBE,
we discretized the intervals $a\in[0,\,15]\,,b\in[0,\,1]$ and $c\in[0,\,5]$
with $\Delta a=\Delta b=\Delta c=0.1$. We note that for faster convergence
the root finding method needs an initial seed close to the steady
state solution. Since we have no information about the existing steady
state, we seed the root finding method with $10$ different uniform
initial guesses i.e., 
\[
\left\{ u_{0}(x)=2^{i}\,|\,i=0,1,\dots,9\right\} \,,
\]
before we conclude a positive steady state does not exist for a given
point $(a,\,b,\,c)$. Consequently, we checked for the existence of
a positive steady state at each of these discrete points. As depicted
in Figure \ref{fig:pbe_region}, the existence region of positive
steady states of the PBE forms a three dimensional wedge like region.
Moreover, in Figures \ref{fig:b_0.1}-\ref{fig:b_1.0}, to deduce
stability of each steady state solution, we checked the spectrum of
the Jacobian matrix evaluated at each steady state. Particularly,
if the real part of rightmost eigenvalue of the Jacobian matrix is
negative, the steady state is identified as locally stable (blue region).
Conversely, if the real part of rightmost eigenvalue of the Jacobian
matrix is positive the steady state is identified as unstable (red
region). One can observe that growth ($b$) and removal ($c$) rates
can balance the smaller renewal rates ($a$), and thus locally stable
steady states exist. However, as the renewal rate gets larger steady
states first become unstable and then cease to exist (yellow region).
This is also illustrated in Figure \ref{fig:effect_of_renewal}, where
steady states start diverging for the larger renewal rates ($a$).

\begin{figure}
\centering{}\subfloat[\label{fig:pbe_region}]{\protect\centering{}\protect\includegraphics[width=0.4\paperwidth]{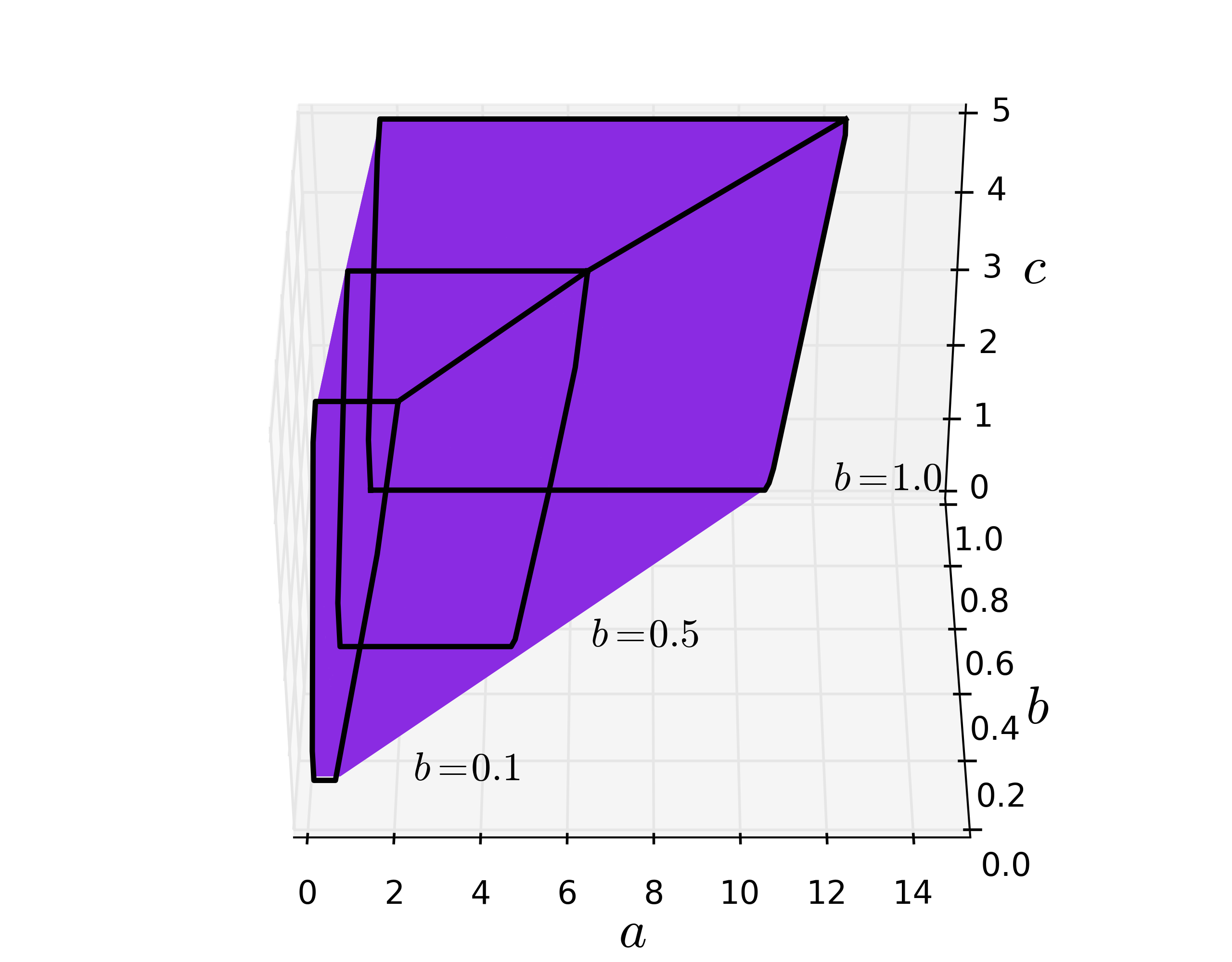}\protect}~\subfloat[\label{fig:b_0.1}]{\protect\centering{}\protect\includegraphics[width=0.35\paperwidth]{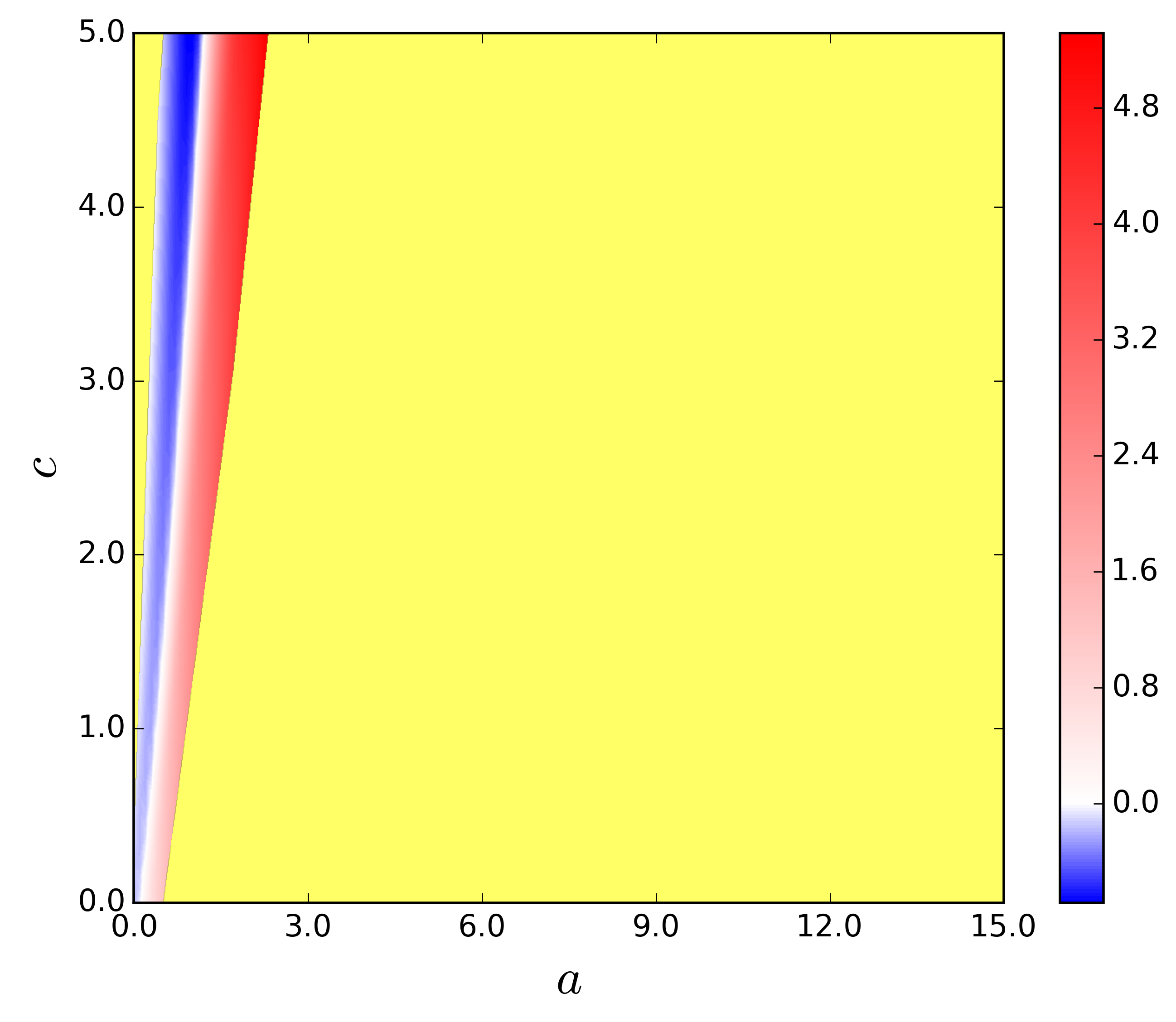}\protect}\\
\subfloat[]{\protect\centering{}\protect\includegraphics[width=0.35\paperwidth]{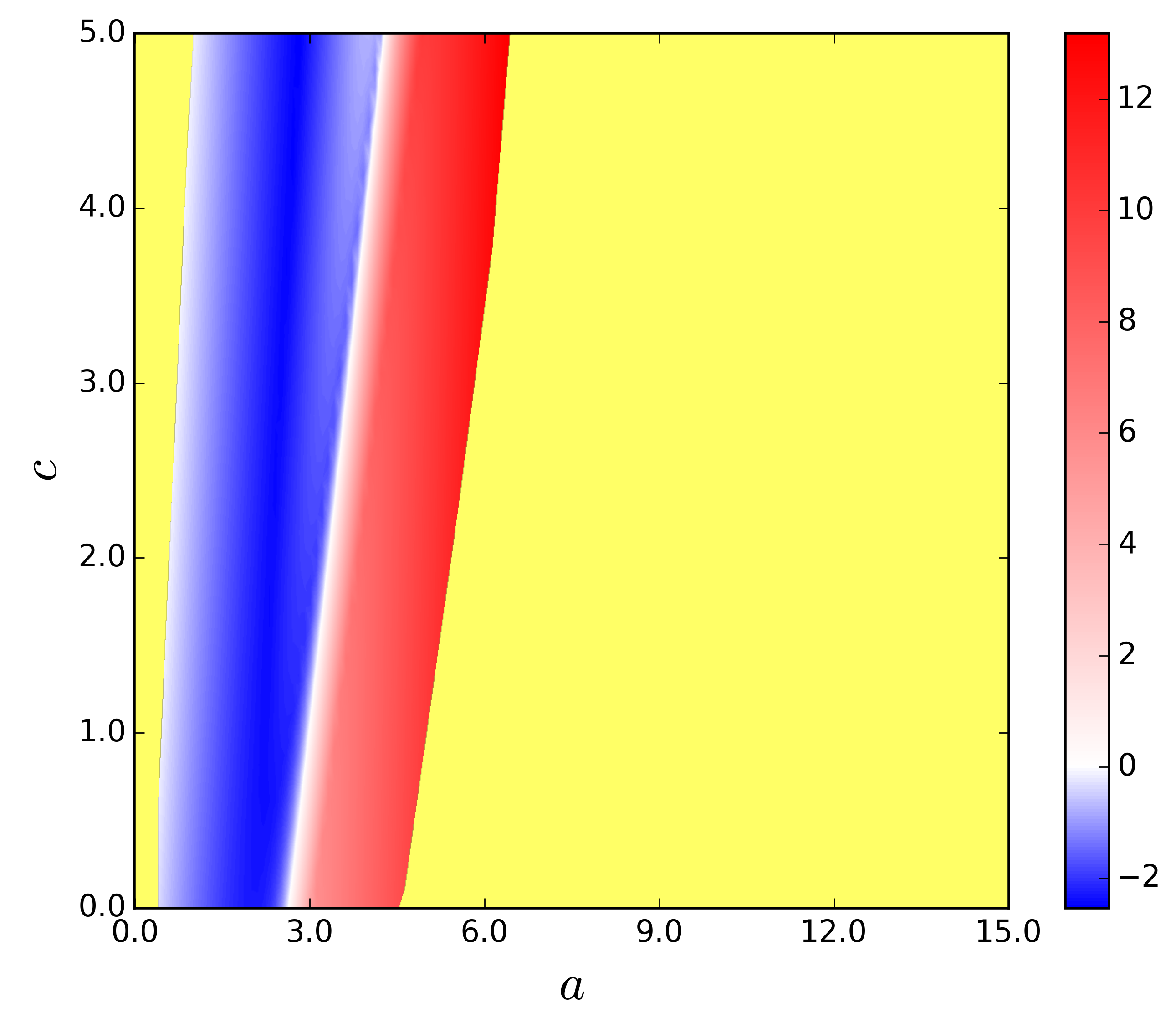}\protect}\hspace*{1.4cm}\subfloat[\label{fig:b_1.0}]{\protect\centering{}\protect\includegraphics[width=0.35\paperwidth]{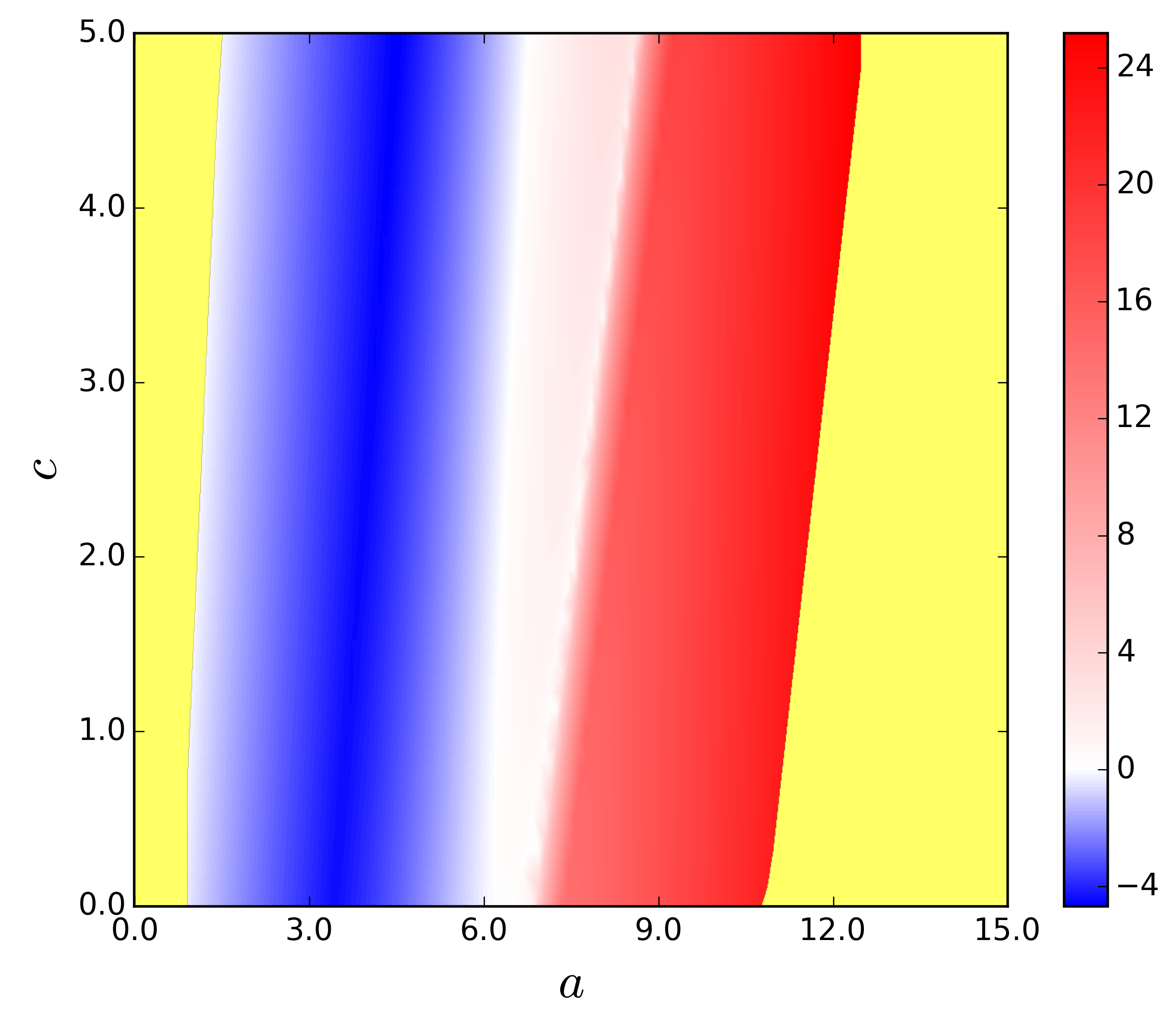}\protect}\protect\caption{Existence and stability regions for the steady states of the PBE a)
Existence region for the steady states of the PBE forms a wedge like
shape. b) Stability region for $b=0.1$, $a\in[0,\,15]$ and $c\in[0,\,5]$.
c) Stability region for $b=0.5$, $a\in[0,\,15]$ and $c\in[0,\,5]$.
d) Stability region for $b=1.0$, $a\in[0,\,15]$ and $c\in[0,\,5]$.
Color bar represents the real part of rightmost eigenvalue of the
Jacobian matrix evaluated at each steady state. Yellow regions represents
the region for which a positive steady state does not exists. }
\end{figure}

\begin{figure}
\centering{}\subfloat[\label{fig:Steady-state-solution}]{\protect\centering{}\protect\includegraphics[scale=0.4]{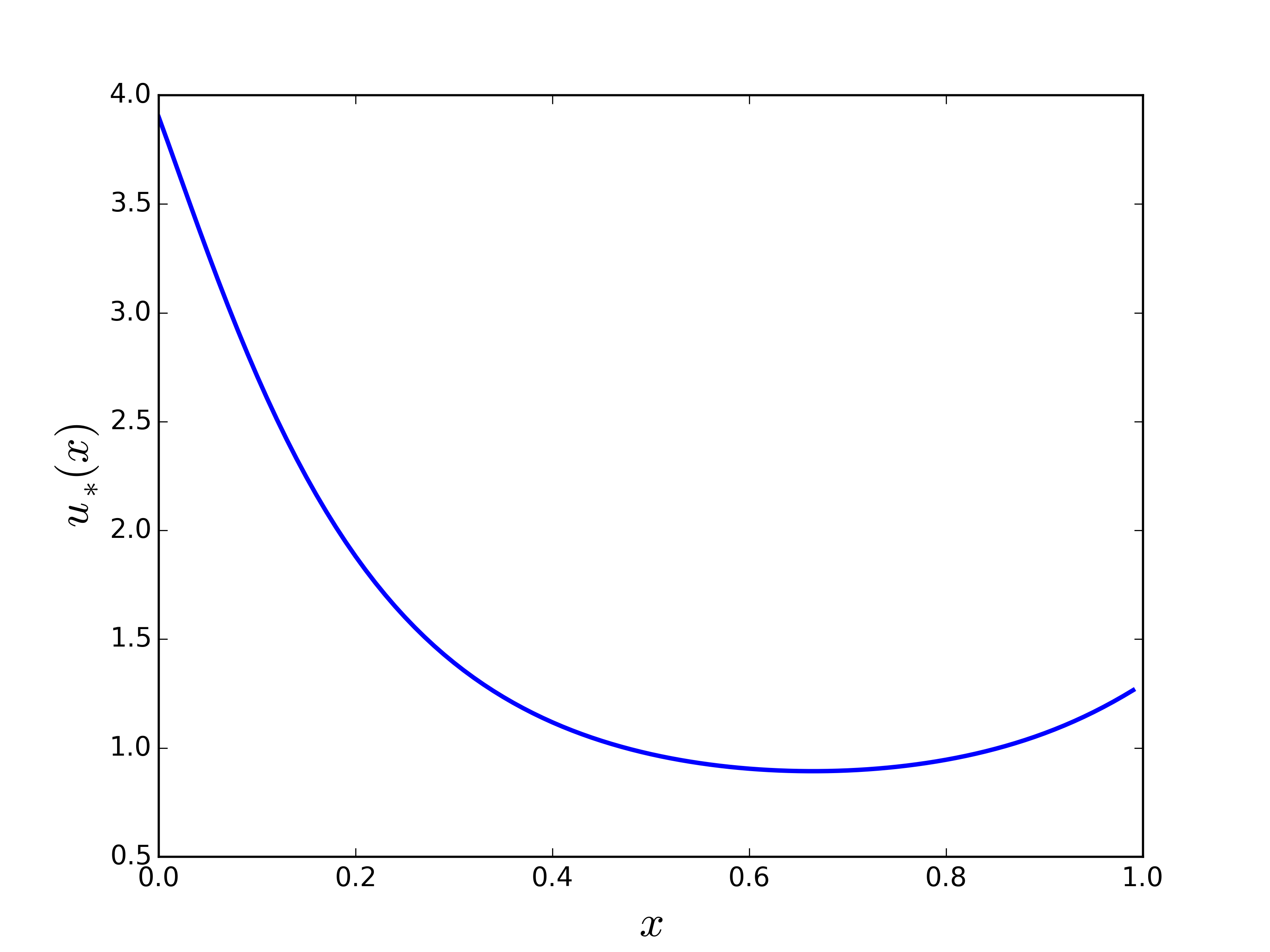}\protect}~\subfloat[\label{fig:effect_of_renewal}]{\protect\centering{}\protect\includegraphics[scale=0.4]{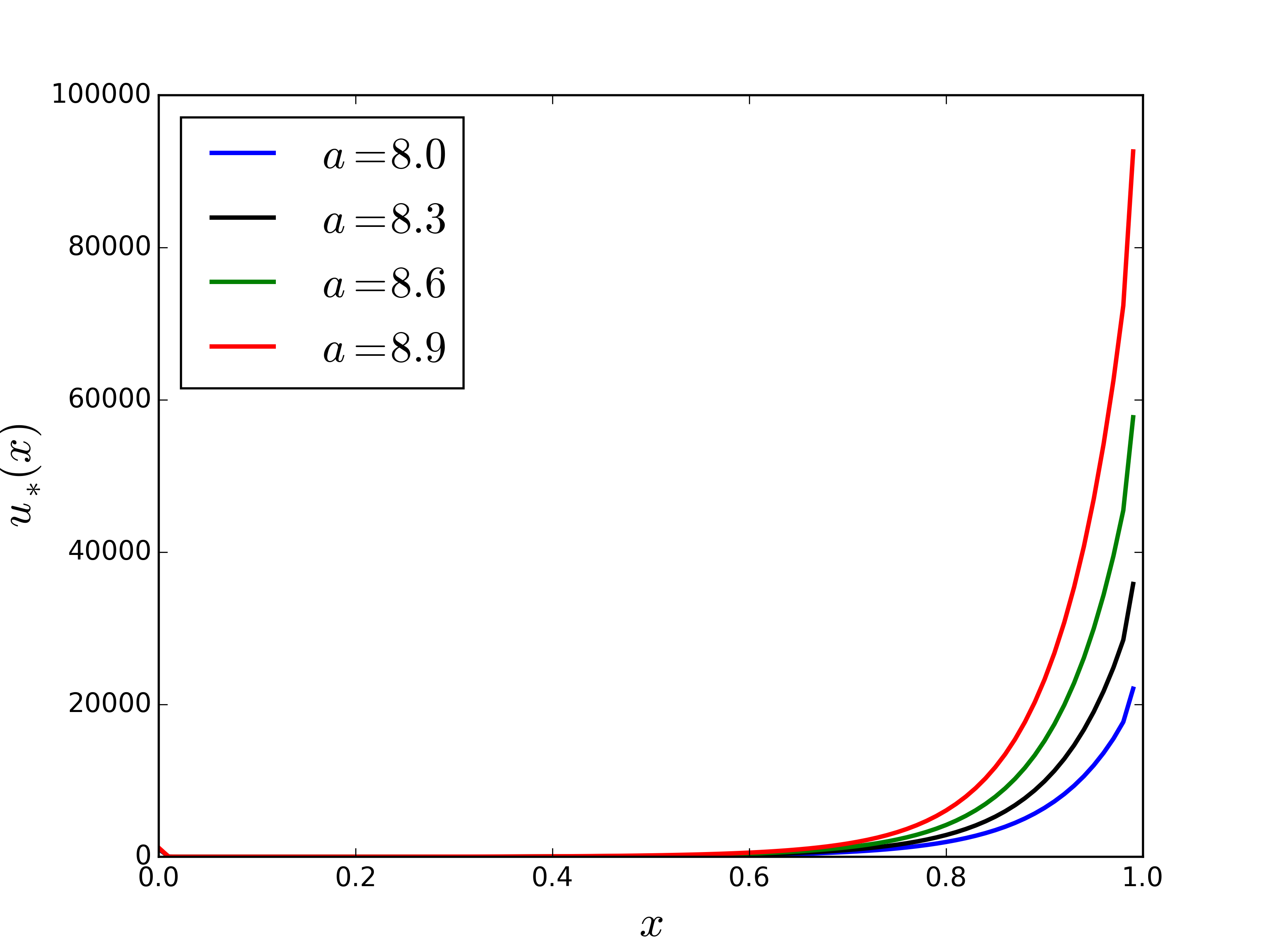}\protect}\protect\caption{a) An example steady-state solution of the PBE for $b=0.5,\,a=c=1$.
b) Steady states for increasing renewal rate and $b=c=1$}
\end{figure}

Figure \ref{fig:Steady-state-solution} illustrates an example stationary
solution for $b=0.5,\,a=c=1$. To confirm that the function depicted
in Figure \ref{fig:Time-evolution-of} is indeed a locally stable
steady state, we simulated the system of ODEs in (\ref{eq:numerical approximation})-(\ref{eq:initial condition projection})
for $t\in[0,\,10]$ with a collection of arbitrary initial conditions
(Figure \ref{fig:Time-evolution-of}a) close to the steady state solution.
One can observe in Figure \ref{fig:Time-evolution-of} that the stationary
solution is indeed locally stable and thus initial conditions, Figure
\ref{fig:Time-evolution-of}, converge to the steady state depicted
in Figure \ref{fig:Time-evolution-of}b. As depicted in Figures \ref{fig:Time-evolution-of}c
and \ref{fig:Time-evolution-of}d, convergence is also reflected in
the evolution of the total number of flocs (zeroth moment),
\[
M_{0}(t)=\int_{0}^{\overline{x}}u(t,\,x)\,dx\approx\sum_{i=1}^{n}\int_{x_{i-1}^{n}}^{x_{i}^{n}}\alpha_{i}\beta_{i}^{n}(x)\,dx=\Delta x\sum_{i=1}^{n}\alpha_{i}\,,
\]
and total mass of the flocs (first moment),
\[
M_{1}(t)=\int_{0}^{\overline{x}}xu(t,\,x)\,dx\approx\sum_{i=1}^{n}\int_{x_{i-1}^{n}}^{x_{i}^{n}}\alpha_{i}x\beta_{i}^{n}(x)\,dx=\frac{\Delta x}{2}\sum_{i=1}^{n}\alpha_{i}\left(x_{i}^{n}+x_{i-1}^{n}\right)\,.
\]

\begin{figure}
\centering{}\subfloat[]{\protect\includegraphics[scale=0.4]{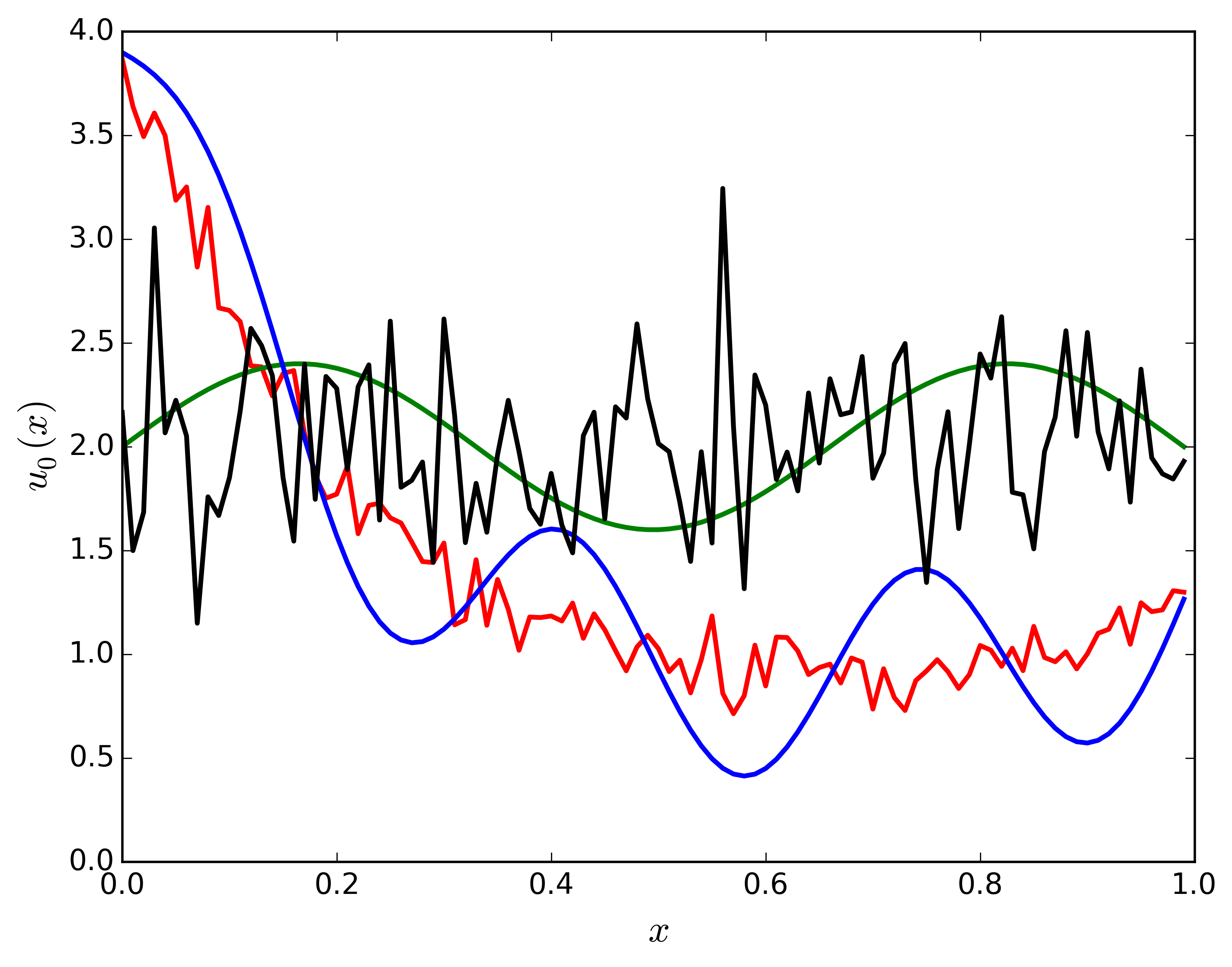}

}~\subfloat[]{\protect\includegraphics[scale=0.4]{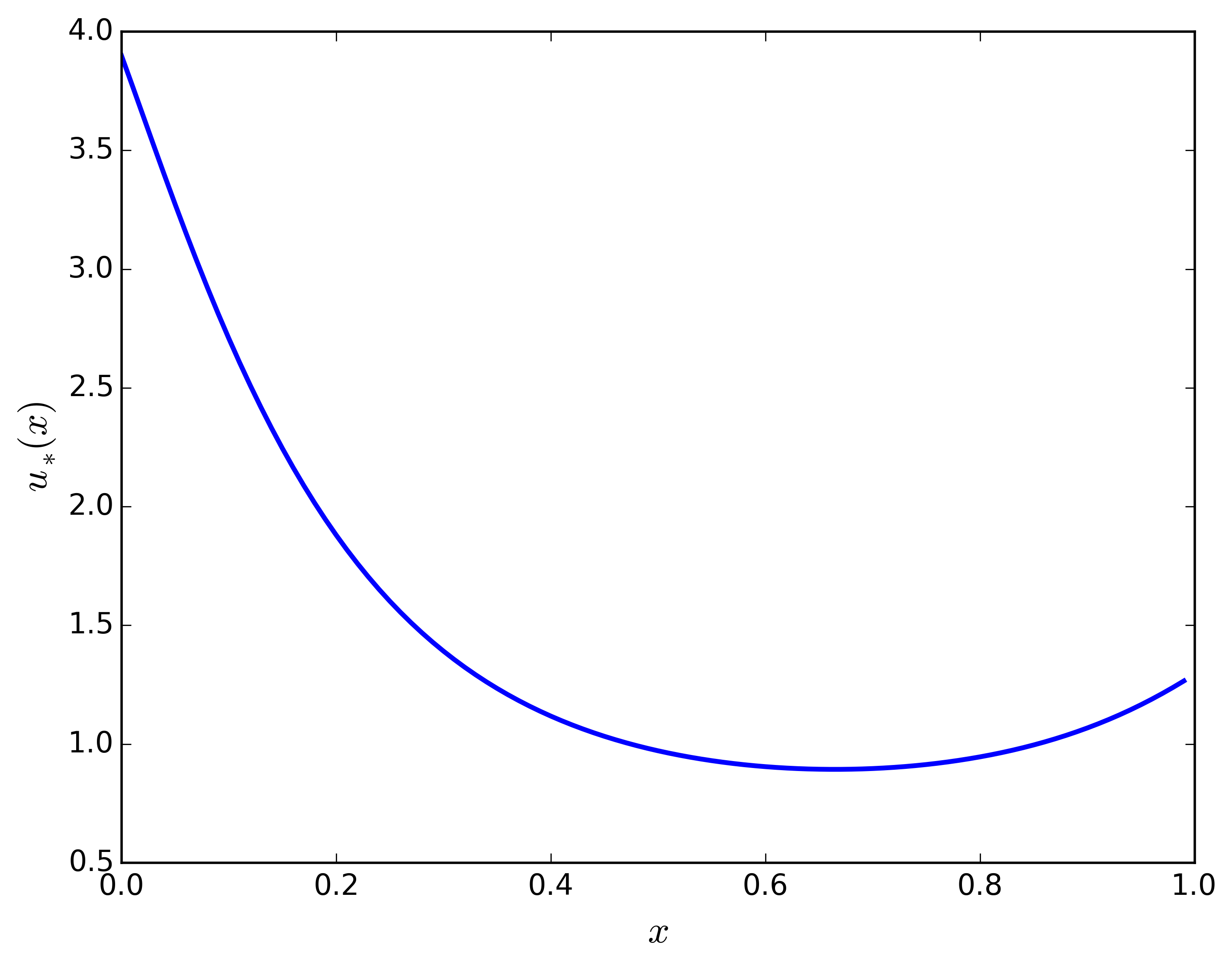}

}\\
\subfloat[]{\protect\includegraphics[scale=0.4]{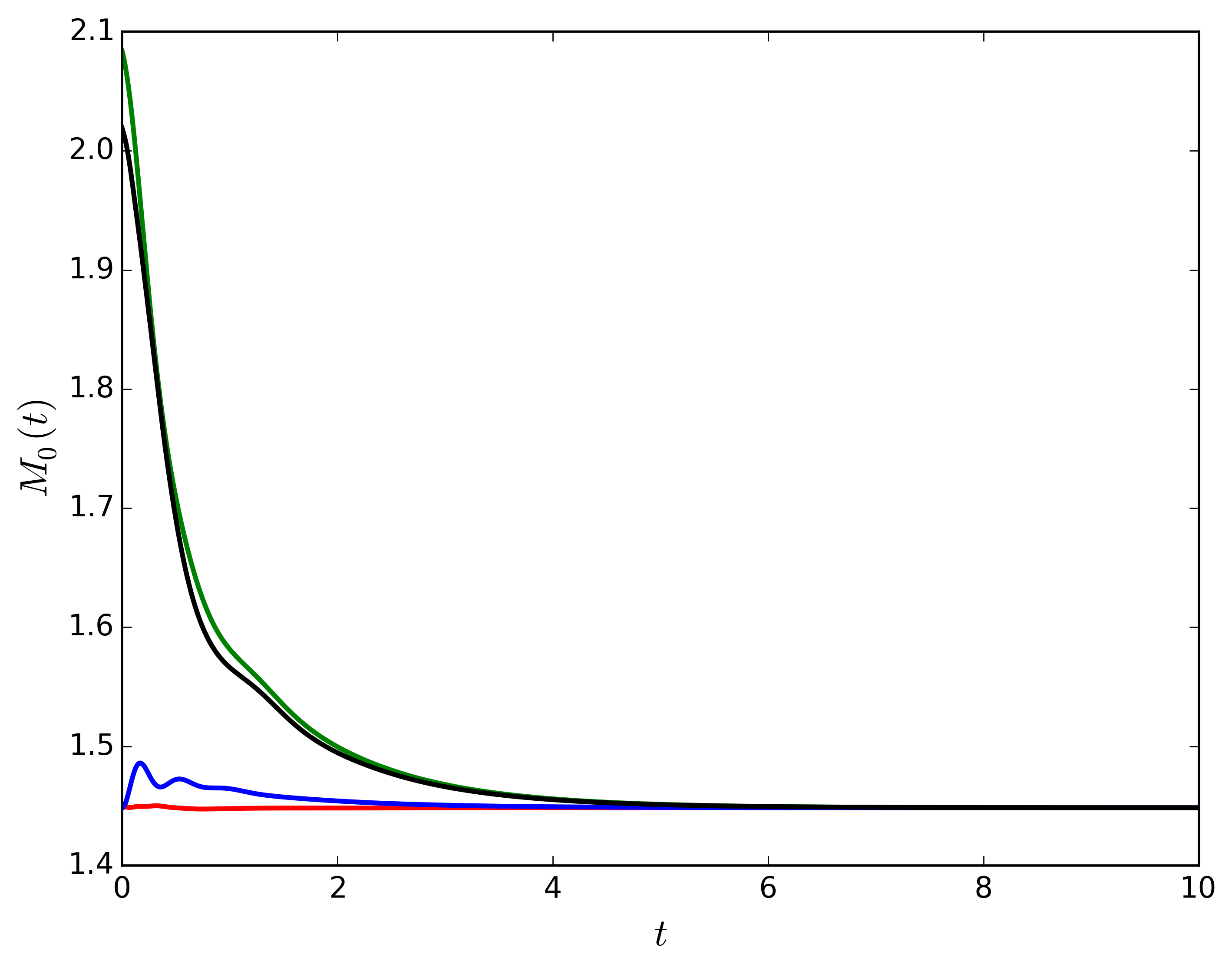}

}~\subfloat[]{\protect\includegraphics[scale=0.4]{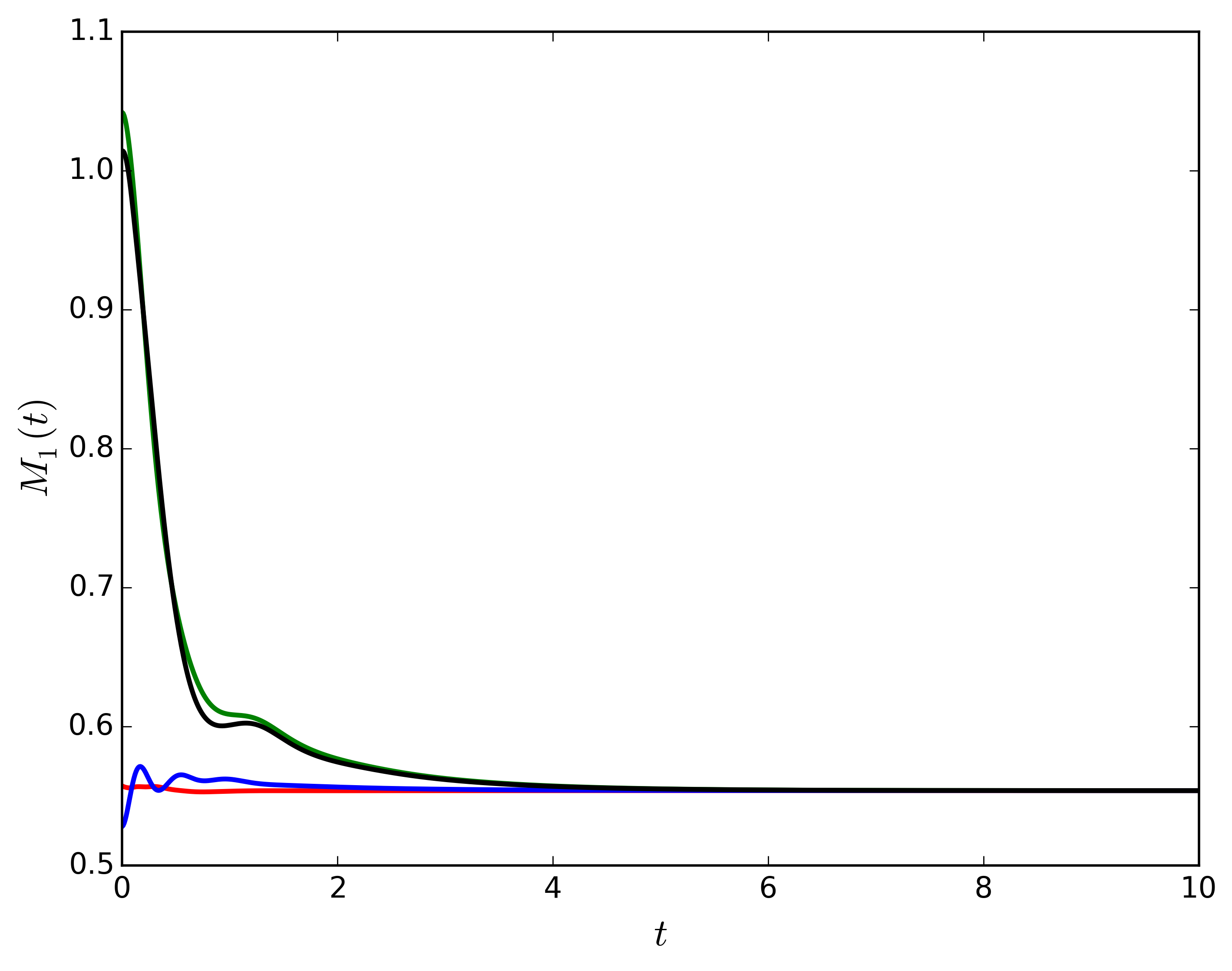}

}\protect\caption{\label{fig:Time-evolution-of}Time evolution of the flocculation model
with arbitrary initial conditions. a) Four different initial conditions
are chosen close to the steady state. b) Solution of the PBE for those
initial conditions at $t=10$. c) Evolution of the total number $M_{0}(t)$
of the flocs for $t\in[0,\,10]$. d) Evolution of the total mass $M_{1}(t)$
of the flocs for $t\in[0,\,10]$.}
\end{figure}
Moreover, to confirm that the steady state solution is not changing
with increasing dimension of approximate spaces $\mathcal{X}_{n}$,
we simulated our numerical scheme for different values of $n$. Figure
\ref{fig:Change-in-zeroth}a illustrates that stationary solutions
converge to exact stationary solution of (\ref{eq: agg and growth model})
as $n\to\infty$. Furthermore, one can observe, in Figure (\ref{fig:Change-in-zeroth})b
that difference between approximate steady states for different values
of $n$ is considerably small.

\begin{figure}
\centering{}\subfloat[]{\protect\centering{}\protect\includegraphics[scale=0.5]{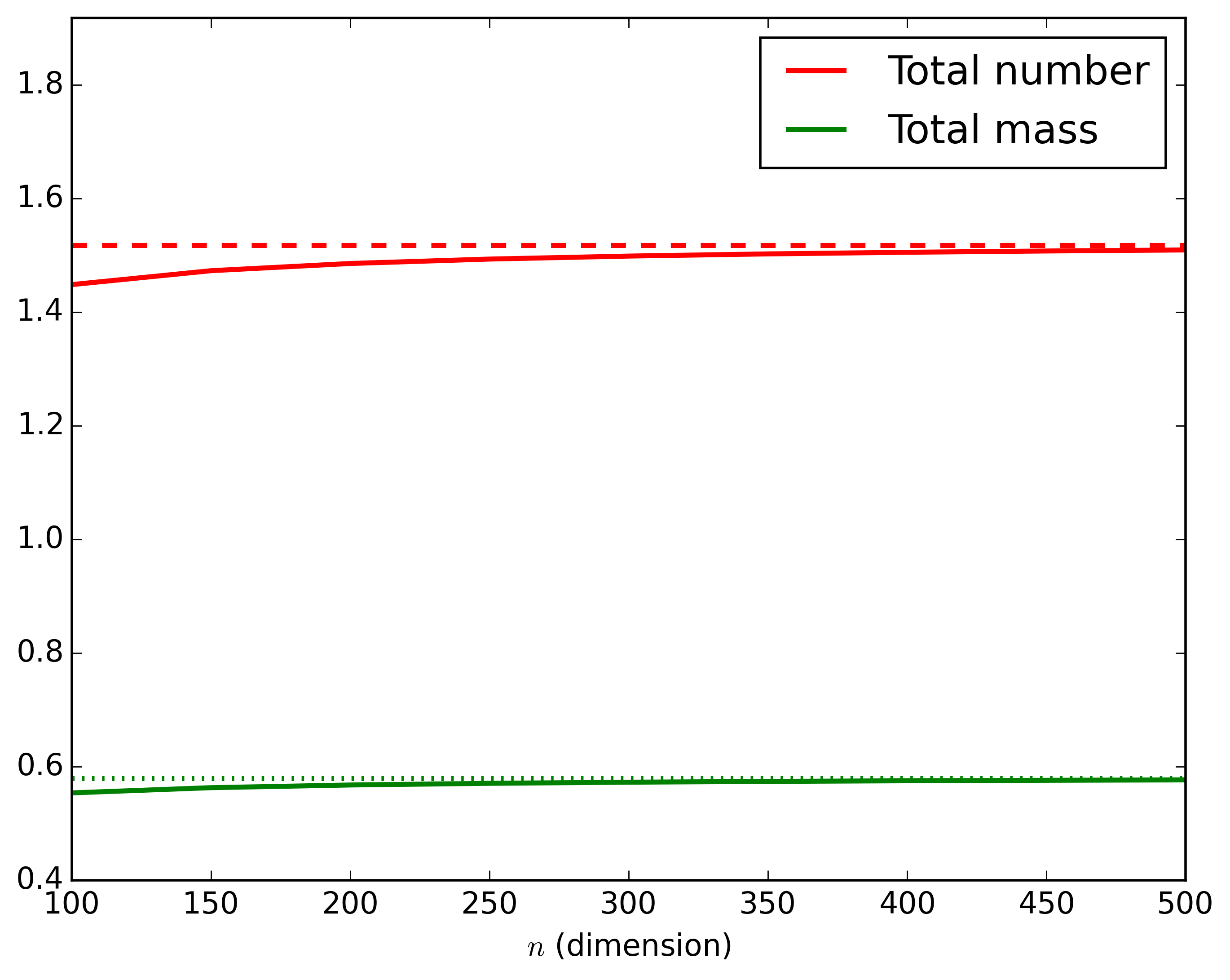}\protect}~\subfloat[]{\protect\begin{centering}
\protect\includegraphics[scale=0.5]{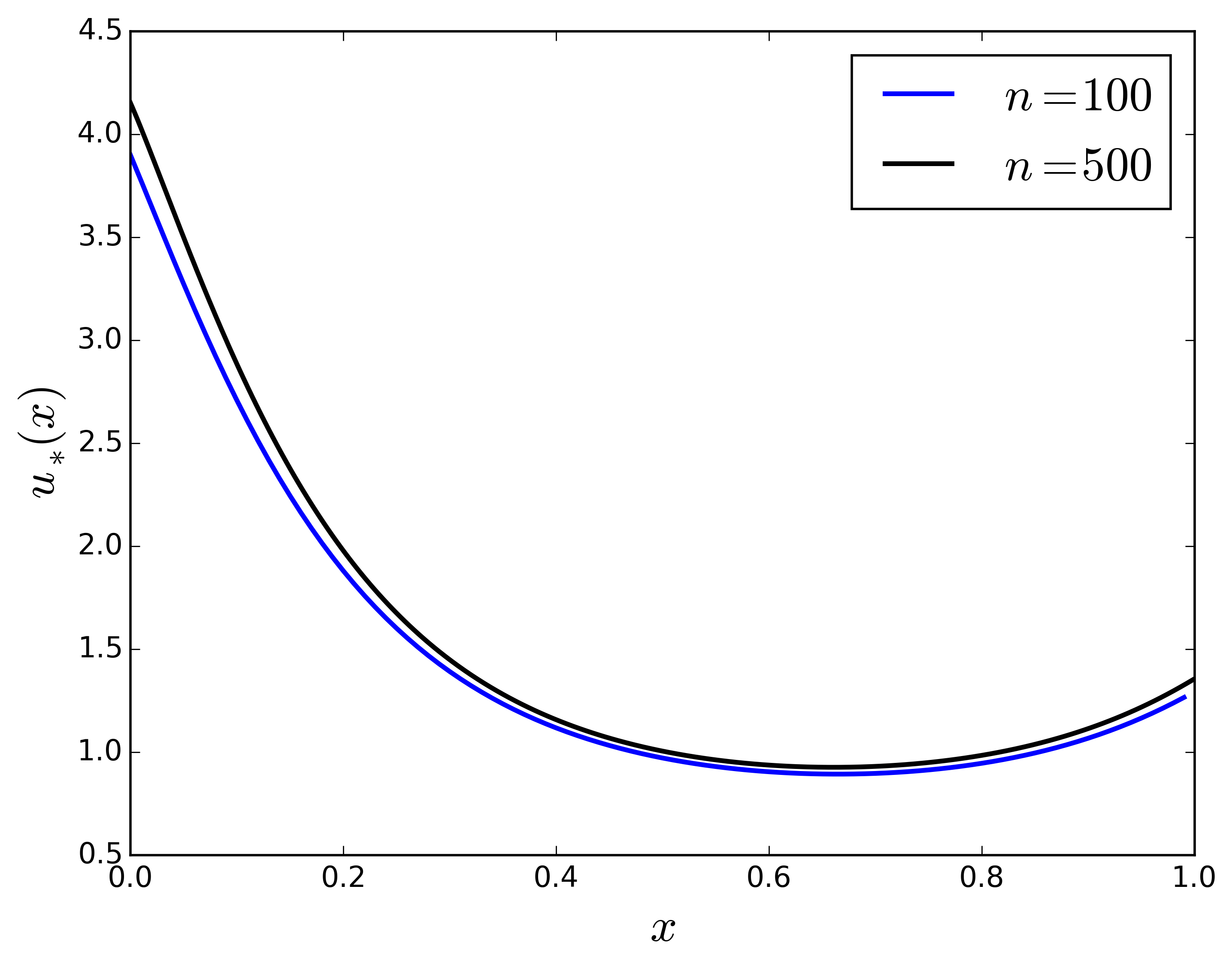}\protect
\par\end{centering}

}\protect\caption{\label{fig:Change-in-zeroth}Change in zeroth and first moments with
increasing dimension of the approximate space $\mathcal{X}_{n}$.
a) Change in the total number and the total mass of the flocs with
respect to increasing dimension $n$. Dashed red lines and dotted
green lines corresponds to the total number and the total mass of
the flocs of the steady state for $n=1000$, respectively. b) Steady
state solution for $n=100$ and $n=500$. }
\end{figure}

\section{\label{sec:Establishing-local-stability}Stability of stationary
solutions}

Studying the asymptotic behavior of solutions is a fundamental tool
for exploring the evolution equations which arise in the mathematical
modeling of real world phenomena. To this end, many mathematical methods
has been developed to describe long-term behavior of evolution equations.
For instance, for long-time behavior of linear evolution equations,
linear semigroup theoretic methods can be used to formulate physically
interpretable conditions. Furthermore, for nonlinear evolution equations,
the principle of linearized stability can be used to relate the spectrum
of the linearized infinitesimal generator to the local stability or
instability of the stationary solution. Nevertheless, investigating
the spectrum of the linearized infinitesimal generator is cumbersome
and requires advanced functional analysis techniques. In contrast
to general evolution equations, the asymptotic behavior of ordinary
differential equations are determined by the eigenvalues of the Jacobian
and well-studied. Hence, in this section we demonstrate that the approximation
scheme presented in Section \ref{sec:Theoretical-Framework} can also
be used for deriving stability conditions for stationary solutions
of the general evolution equations. Towards this end, we prove the
following stability result for general evolution equations.
\begin{cor}
\label{thm:asymptotic stability }Let $u_{*}$ denote a stationary
solution of the abstract evolution (\ref{eq:abstract evolution equation})
and $J_{A}(u_{n})$ denote the Jacobian of the approximate system
of ODEs defined in (\ref{eq:approximate IVP}) . If for all sufficiently
large $n$ the eigenvalues of $J_{A}(P_{n}u_{*})$ are strictly negative,
then $u^{*}$ is a locally asymptotically stable stationary solution
of the evolution equation (\ref{eq:abstract evolution equation}).
In particular, for every closed finite time interval $[0,\,t_{f}]$
and $\epsilon>0$, there exists $\delta>0$ such that a unique solution
of (\ref{eq:abstract evolution equation}), $u(t,\,x)$, with initial
condition $u_{0}$ fulfilling $\left\Vert u_{0}-u_{*}\right\Vert <\delta$
satisfies 
\begin{equation}
\left\Vert u(t,\,\cdot)-u_{*}\right\Vert <\epsilon\label{eq:stability inequality}
\end{equation}
for all $t\in[0,\,t_{f}]$. \end{cor}
\begin{proof}
Since the infinitesimal generator approximation scheme, presented
in Section \ref{sub:Approximation-scheme}, is convergent, for every
given $\epsilon>0$ and finite time interval $[0,\,t_{f}]$ there
exist $n_{\epsilon}\in\mathbb{N}$ such that for $n\ge n_{\epsilon}$,
\begin{equation}
\left\Vert u(t,\,\cdot)-E_{n}u_{n}(t)\right\Vert <\epsilon/2\label{eq:smth one}
\end{equation}
for all $t\in[0,\,t_{f}]$ (where the bounded linear function $E_{n}$
is defined as in (\ref{eq:reals to space})). Moreover, the eigenvalues
of $J_{A}(P_{n}u_{*})$ are strictly negative for all sufficiently
large $n$. This in turn implies that $P_{M}u_{*}$ is a locally asymptotically
stable solution of (\ref{eq:approximate IVP}) for some $M\ge n_{\epsilon}$.
That is, for given $\epsilon>0$ there is $\delta>0$ such that 
\begin{equation}
\left\Vert u_{M}(t,\,\cdot)-P_{M}u_{*}\right\Vert _{\mathbb{R}^{M}}=\left\Vert E_{M}u_{M}(t,\,\cdot)-u_{*}\right\Vert _{\mathcal{X}}<\epsilon/2\label{eq:smth two}
\end{equation}
for all $t>0$ and for all $u_{0}$ such that $\left\Vert P_{M}u_{0}-P_{M}u_{*})\right\Vert _{\mathbb{R}^{M}}=\left\Vert u_{0}-u_{*}\right\Vert _{\mathcal{X}}<\delta$
(see for instance \citep[\S 23]{arnold1992ordinary}). Consequently,
combining (\ref{eq:smth one}) and (\ref{eq:smth two}) yields 
\[
\left\Vert u(t,\,\cdot)-u_{*}\right\Vert \le\left\Vert u(t,\,x)-E_{M}u_{M}(t,\,\cdot)\right\Vert +\left\Vert E_{M}u_{M}(t,\,\cdot)-u_{*}\right\Vert <\epsilon
\]
for all $t\in[0,\,t_{f}]$ and for all $u_{0}$ such that $\left\Vert u_{0}-u_{*}\right\Vert <\delta$.
\end{proof}
Having the required corollary in hand, in subsequent sections we apply
it to two different examples from size-structured population modeling.
In Section \ref{sub:Sinko-Streifer-model}, we derive conditions for
local stability of the stationary solutions of linear Sinko-Streifer
equation (\ref{eq:sinko-streifer}). In Section \ref{sub:Population-balance-equation},
we derive conditions for local stability of the stationary solution
of the nonlinear population balance equation defined in (\ref{eq: agg and growth model}).

\subsection{\label{sub:Sinko-Streifer-model}Sinko-Streifer model}

To prove the first statement of Corollary \ref{thm:asymptotic stability },
we use the well-known Gershgorin theorem for locating eigenvalues
of a matrix. The Gershgorin theorem states that each eigenvalue of
$A$ lies in one of the the circular areas, called Gershgorin disks,
in the complex plane that is centered at the $i$th diagonal element
and whose radius is $R_{i}$,
\[
\left\{ z\in\mathbb{C}\,:\,\left|z-a_{ii}\right|\le R_{i}\right\} \,,
\]
where
\[
R_{i}=\sum_{j=1,\,j\ne i}^{n}|a_{ji}|\,.
\]
Consequently, by applying the Gershgorin theorem to columns of the
matrix $\mathcal{G}_{n}$, the approximate ODE system for Sinko-Streifer
derived in Section \ref{sub:Numerical-convergence-results}, yields
that its eigenvalues are located in the following Gershgorin disks
\[
\left\{ z\in\mathbb{C}\,:\,\left|z+\frac{1}{\Delta x}g(x_{i}^{n})+\mu(x_{i}^{n})\right|\le q(x_{i}^{n})+\frac{1}{\Delta x}g(x_{i}^{n})\right\} \,
\]
for $i=1,\cdots,n$. Therefore, a sufficient condition for eigenvalues
of the matrix $\mathcal{G}_{n}$ to be strictly negative is 
\[
q(x_{i}^{n})<\mu(x_{i}^{n})
\]
for each $i=1,\cdots,n$. Hence, the condition 
\[
q(x)-\mu(x)<0
\]
for all $x\in(0,\,\overline{x})$ ensures the local stability of the
non-trivial stationary solution of Sinko-Streifer equation, which
is also in an agreement with the stability conditions derived in \citep[Condition 2]{mirzaev2015stability}.
We summarize the results of this subsection in the following proposition.
\begin{prop}
Let $u_{*}$ be a stationary solution of the Sinko-Streifer model
defined in (\ref{eq:sinko-streifer}). Moreover, assume that 

\begin{equation}
q(x)-\mu(x)<0\label{eq:first condition-1}
\end{equation}
 for all $x\in[0,\,\overline{x}]$, then the stationary solution $u_{*}$
is locally stable in the sense of Corollary \ref{thm:asymptotic stability }. 
\end{prop}

\subsection{\label{sub:Population-balance-equation}Population balance equation}

Since the approximate system for the microbial flocculation model
is nonlinear, we linearize the system around its stationary solutions.
Let $u_{*}\in L^{1}(0,\,\overline{x})$ be a stationary solution of
(\ref{eq: agg and growth model}) and denote the projection of the
stationary solution $u_{*}$ onto $\mathbb{R}^{n}$ by $\alpha=P_{n}u_{*}=[\alpha_{1},\cdots,\alpha_{n}]^{T}$,
then the Jacobian of the approximate operator $\mathcal{F}_{n}$ defined
in (\ref{eq:numerical approximation}) can be written as
\begin{equation}
J_{\mathcal{F}}(\alpha)=\mathcal{G}_{n}+J_{\mathcal{A}}(\alpha)+J_{\mathcal{B}}(\alpha)\,,\label{eq:jacobian of F}
\end{equation}
where $\mathcal{G}_{n}$ is defined in (\ref{eq:approximate G}),
\begin{alignat*}{1}
J_{\mathcal{A}}(\alpha) & =\begin{pmatrix}-\alpha_{1}k_{a}(x_{1}^{n},\,x_{1}^{n})\Delta x & -\alpha_{1}k_{a}(x_{1}^{n},\,x_{2}^{n})\Delta x & \cdots & -\alpha_{1}k_{a}(x_{1}^{n},\,x_{n-1}^{n})\Delta x & 0\\
-\alpha_{2}k_{a}(x_{2}^{n},\,x_{1}^{n})\Delta x & \cdots & -\alpha_{2}k_{a}(x_{2}^{n},\,x_{n-2}^{n})\Delta x & 0 & 0\\
\vdots & \vdots & \ddots & \vdots & \vdots\\
-\alpha_{N-1}k_{a}(x_{n-1}^{n},\,x_{1}^{n})\Delta x & 0 & \cdots & 0 & 0\\
0 & 0 & \cdots & 0 & 0
\end{pmatrix}\\
 & \quad+\begin{pmatrix}-\sum_{j=1}^{n-1}k_{a}(x_{1}^{n},\,x_{j}^{n})\alpha_{j}\Delta x & 0 & \cdots & 0 & 0\\
\alpha_{1}k_{a}(x_{1}^{n},\,x_{1}^{n})\Delta x & -\sum_{j=1}^{n-2}k_{a}(x_{2}^{n},\,x_{j}^{n})\alpha_{j}\Delta x & 0 & \cdots & 0\\
\alpha_{2}k_{a}(x_{1}^{n},\,x_{2}^{n})\Delta x & \alpha_{1}k_{a}(x_{2}^{n},\,x_{1}^{n})\Delta x & \ddots & 0 & \vdots\\
\vdots & \vdots & \ddots & -\sum_{j=1}^{1}k_{a}(x_{n-1}^{n},\,x_{j}^{n})\alpha_{j}\Delta x & 0\\
\alpha_{n-1}k_{a}(x_{1}^{n},\,x_{n-1}^{n})\Delta x & \alpha_{n-2}k_{a}(x_{2}^{n},\,x_{n-2}^{n})\Delta x & \cdots & \alpha_{1}k_{a}(x_{n-1}^{n},\,x_{1}^{n})\Delta x & 0
\end{pmatrix}\,,
\end{alignat*}
and 
\[
J_{\mathcal{B}}(\alpha)=\begin{pmatrix}-\frac{1}{2}k_{f}(x_{1}^{n}) & \Gamma(x_{1}^{n};x_{2}^{n})k_{f}(x_{2}^{n})\Delta x & \Gamma(x_{1}^{n};x_{3}^{n})k_{f}(x_{3}^{n})\Delta x & \cdots & \Gamma(x_{1}^{n};x_{n}^{n})k_{f}(x_{n}^{n})\Delta x\\
0 & -\frac{1}{2}k_{f}(x_{2}^{n}) & \Gamma(x_{2}^{n};x_{3}^{n})k_{f}(x_{3}^{n})\Delta x & \cdots & \Gamma(x_{2}^{n};x_{n}^{n})k_{f}(x_{n}^{n})\Delta x\\
\vdots & 0 & \ddots & \ddots & \vdots\\
0 & \cdots & 0 & -\frac{1}{2}k_{f}(x_{n-1}^{n}) & \Gamma(x_{n-1}^{n};x_{n}^{n})k_{f}(x_{n}^{n})\Delta x\\
0 & 0 & \cdots & 0 & -\frac{1}{2}k_{f}(x_{n}^{n})
\end{pmatrix}
\]

\begin{figure}
\centering{}\subfloat[\label{fig:evals for 10}]{\protect\includegraphics[width=0.4\columnwidth]{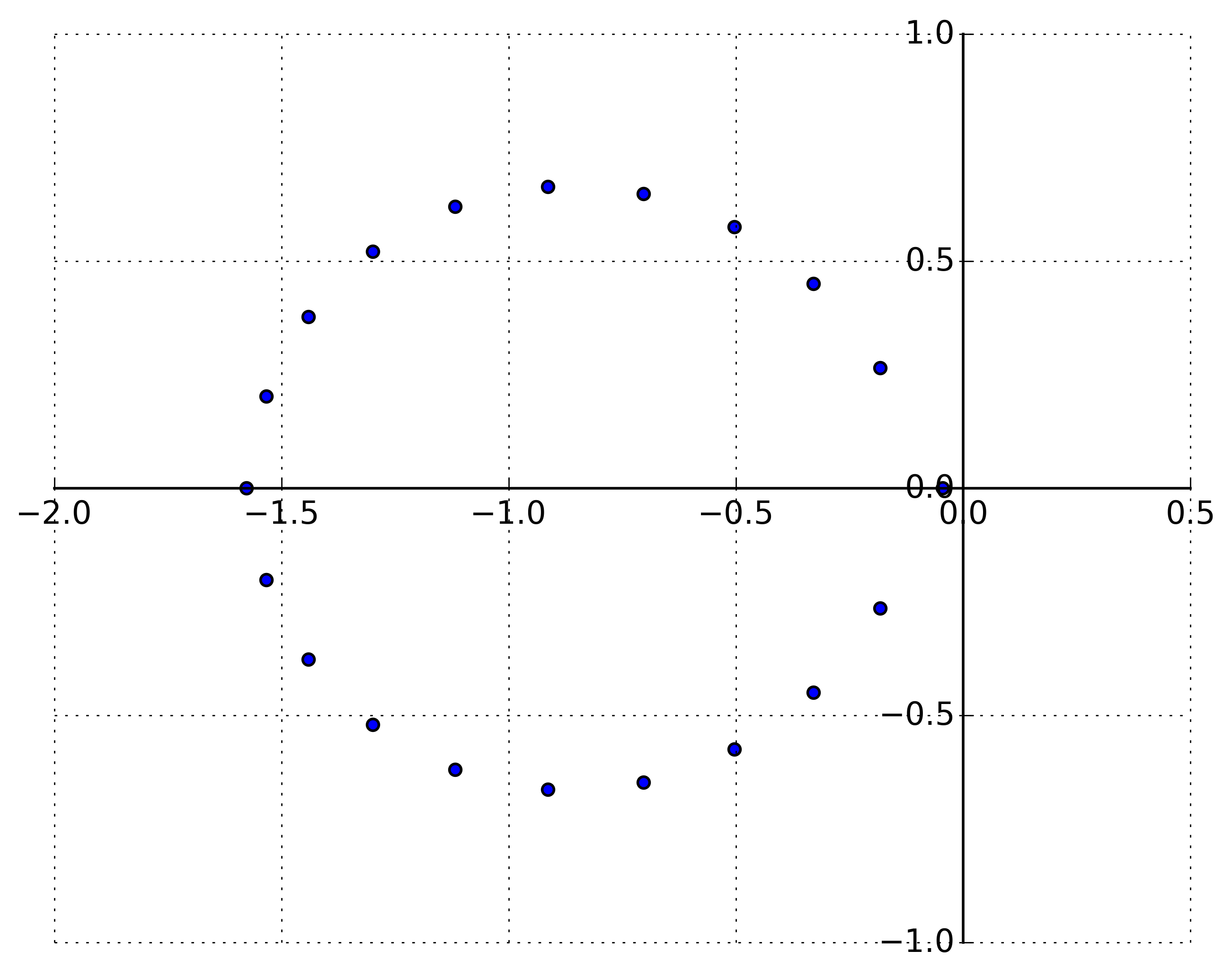}}~\subfloat[\label{fig:evals for 50}]{\protect\includegraphics[width=0.4\columnwidth]{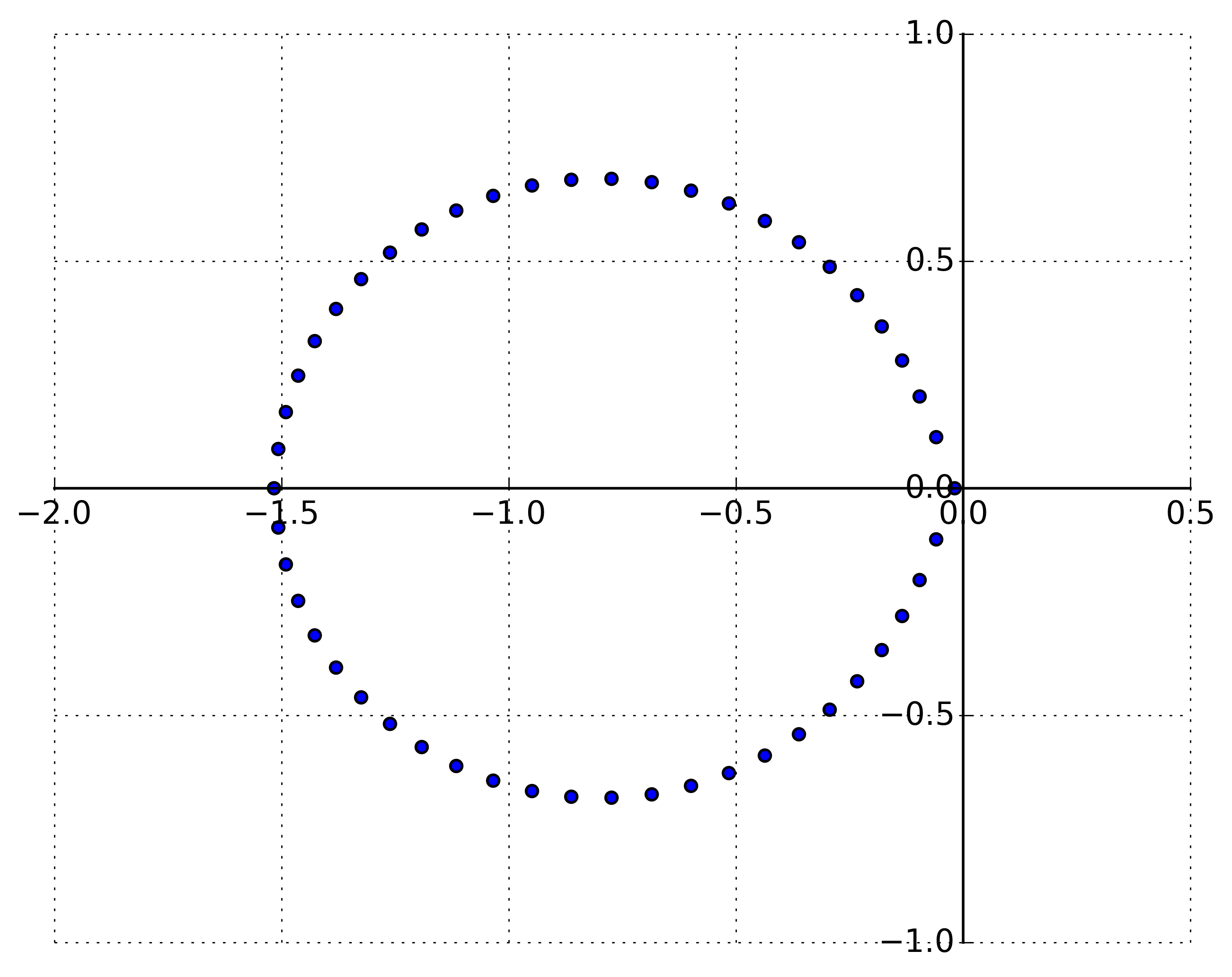}}\\
\subfloat[\label{fig:evals for 200}]{\protect\includegraphics[width=0.4\columnwidth]{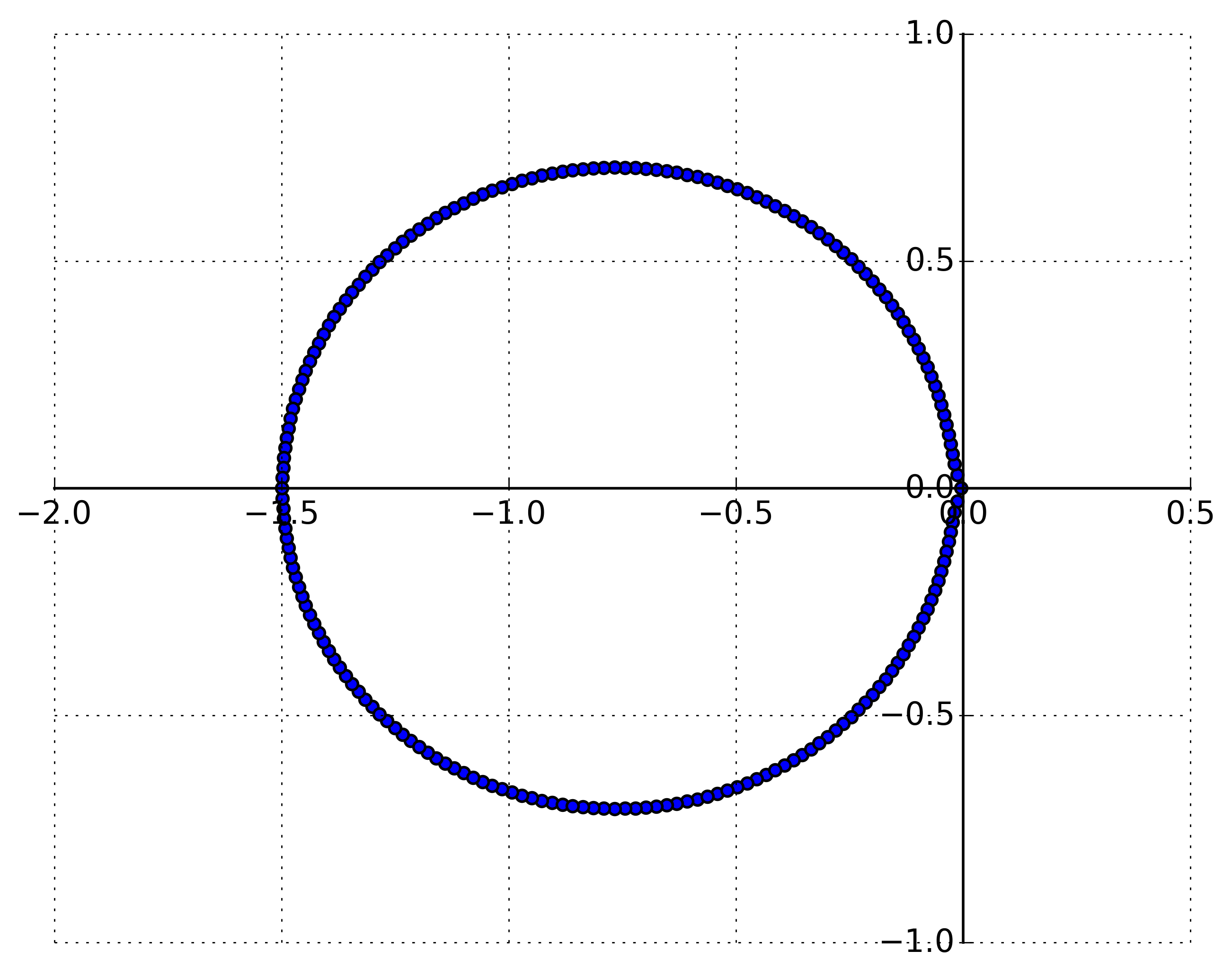}}~\subfloat[\label{fig:convergence of evals}]{\protect\includegraphics[width=0.4\columnwidth]{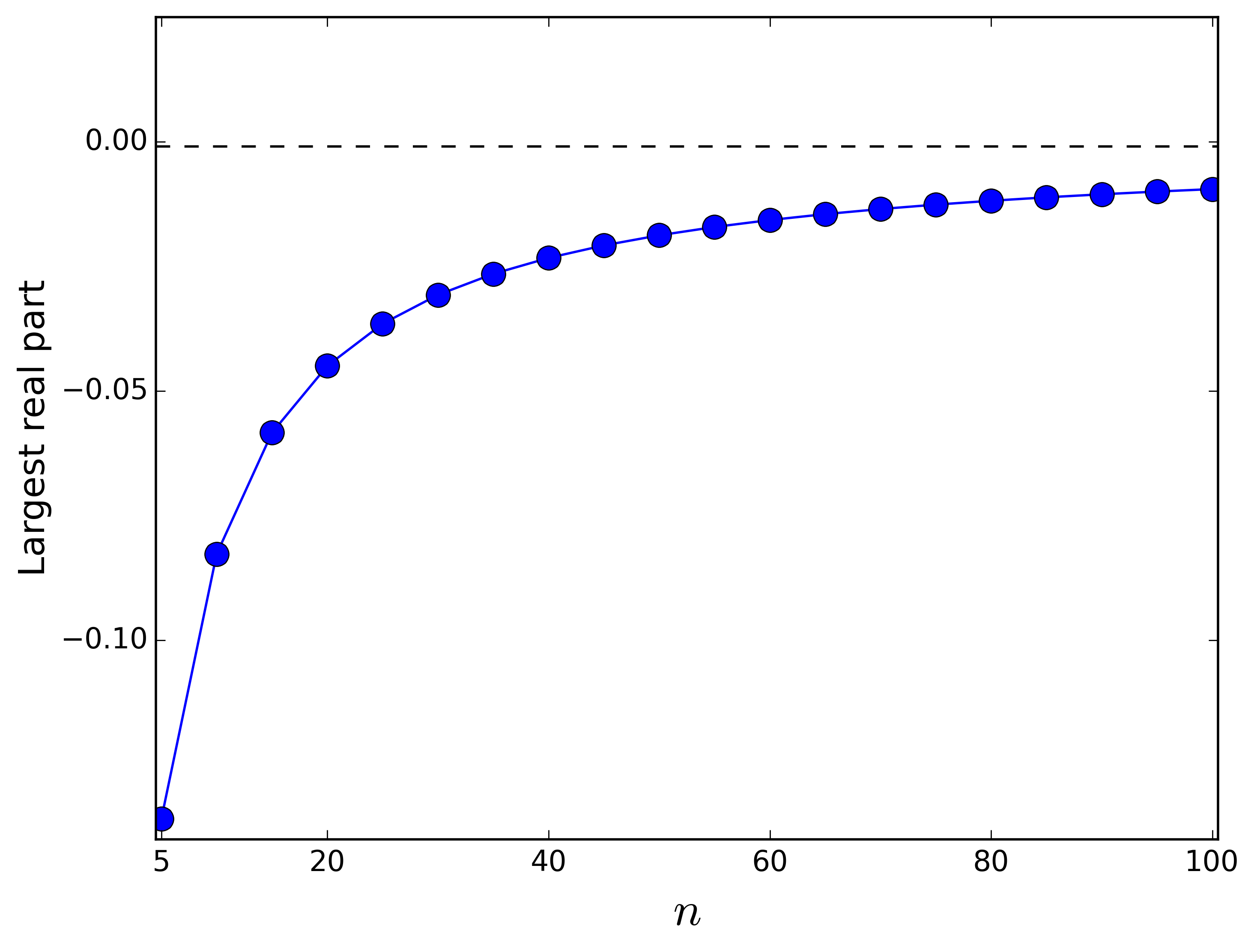}

}\protect\caption{\label{fig:Eigenvalues}Eigenvalues of the Jacobian $J_{\mathcal{F}}(\alpha)$
multiplied by $\Delta x$ for the steady state illustrated in Figure
\ref{fig:Steady-state-solution}. a) Eigenvalues of the Jacobian plotted
in the complex plane for $n=20$. b) Eigenvalues of the Jacobian plotted
in the complex plane for $n=50$. c) Eigenvalues of the Jacobian plotted
in the complex plane for $n=200$. d) Change in the rightmost eigenvalue
for increasing $n$. Dashed black line corresponds to the rightmost
eigenvalue of the Jacobian for $n=1000$.}
\end{figure}

To bound the eigenvalues of $J_{\mathcal{F}}(\alpha)$ again we use
Gershgorin theorem. Consequently, the centers and the radii of Gershgorin
disks are given by 
\[
a_{ii}=-\frac{1}{\Delta x}g(x_{i}^{n})-\mu(x_{i}^{n})-\frac{1}{2}k_{f}(x_{i}^{n})-\alpha_{i}k_{a}(x_{i}^{n},\,x_{i}^{n})\Delta x-\sum_{j=1}^{n-i}k_{a}(x_{i}^{n},\,x_{j}^{n})\alpha_{j}\Delta x
\]
and 
\[
R_{i}\le\frac{1}{\Delta x}g(x_{i}^{n})+q(x_{i}^{n})+\sum_{j=1}^{i-1}\Gamma(x_{j}^{n};x_{i}^{n})k_{f}(x_{i}^{n})\Delta x+\sum_{j=1}^{n-i}\alpha_{j}k_{a}(x_{j}^{n},\,x_{i}^{n})\Delta x+\sum_{j=1,\,j\ne i}^{n-i}\alpha_{j}k_{a}(x_{i}^{n},\,x_{j}^{n})\Delta x\,,
\]
respectively. Consequently, if we can show that 
\begin{equation}
\left|a_{ii}\right|>R_{i}\qquad\text{for each }i\in\left\{ 1,\dots,n\right\} \,,\label{eq:center larger than radius}
\end{equation}
then each of the Gershgorin disks lie strictly on the left side of
the complex plane. To this end, inequality (\ref{eq:center larger than radius})
can be simplified as 
\begin{equation}
\mu(x_{i}^{n})+\frac{1}{2}k_{f}(x_{i}^{n})>q(x_{i}^{n})+\sum_{j=1}^{i-1}\Gamma(x_{j}^{n};x_{i}^{n})k_{f}(x_{i}^{n})\Delta x+\sum_{j=1}^{n-i}\alpha_{j}k_{a}(x_{j}^{n},\,x_{i}^{n})\Delta x\label{eq:approximate inequality}
\end{equation}
for each $i\in\left\{ 1,\cdots,n\right\} $. Accordingly, taking the
limit of (\ref{eq:approximate inequality}) as $n\to\infty$ yields
\begin{equation}
\mu(x)+\frac{1}{2}k_{f}(x)>q(x)+\int_{0}^{x}\Gamma(y,\,x)k_{f}(x)\,dy+\int_{0}^{\overline{x}-x}k_{a}(x,\,y)u_{*}(y)\,dy\label{eq:precise inequality}
\end{equation}
for all $x\in[0,\,\overline{x}]$ and together with the number conservation
requirement (\ref{eq:number conservation}) implies 
\[
q(x)+\frac{1}{2}k_{f}(x)-\mu(x)+\int_{0}^{\overline{x}-x}k_{a}(x,\,y)u_{*}(y)\,dy<0
\]
for all $x\in[0,\,\overline{x}]$. Conversely, note that the integral
approximations in (\ref{eq:approximate inequality}) are right Reimann
sums. Therefore, if the functions $\Gamma(y,\,x)$ and $k_{a}(x,\,y)u_{*}(y)$
are decreasing in $y$ then integral approximations in (\ref{eq:approximate inequality})
are under-approximations of the integrals in (\ref{eq:precise inequality}).
Thus, the inequality stated in (\ref{eq:precise inequality}) ensures
that the eigenvalues of the Jacobian $J_{\mathcal{F}}(\alpha)$ are
strictly negative for all sufficiently large $n$. Now, we are in
a position to summarize the results of this section in the following
proposition. 
\begin{prop}
\label{prop:stability conditions flocculation model}Let $u_{*}$
be a stationary solution of the microbial flocculation model (\ref{eq: agg and growth model}).
Moreover, if 

\begin{equation}
q(x)+\frac{1}{2}k_{f}(x)-\mu(x)+\int_{0}^{\overline{x}-x}k_{a}(x,\,y)u_{*}(y)\,dy<0\label{eq:first condition}
\end{equation}
 for all $x\in[0,\,\overline{x}]$ and 

\begin{equation}
\partial_{y}\left(k_{a}(x,\,y)u_{*}(y)\right)\le0\text{ and }\partial_{y}\Gamma(y,\,x)\le0\label{eq:second condition}
\end{equation}
for all $x\in[0,\,\overline{x}]$ and $y\in[0,\,\overline{x}]$, then
stationary solution $u_{*}$ is locally stable in the sense of Corollary
\ref{thm:asymptotic stability }. 
\end{prop}
To illustrate the utility of the framework developed in this section
we applied our approach to the model rates given in Section \ref{sub:Numerical-implementation-and}
for generation of Figure \ref{fig:Steady-state-solution}. The Beta
distribution used for the post-fragmentation function $\Gamma$ is
not monotonically decreasing, and thus it does not satisfy the conditions
of Proposition \ref{prop:stability conditions flocculation model}.
However, Figure \ref{fig:Eigenvalues}a-c illustrates that the model
rates satisfy the conditions of Corollary \ref{thm:asymptotic stability }.
In fact, Figure \ref{fig:convergence of evals} depicts that for the
steady state illustrated in Figure \ref{fig:Steady-state-solution}
the eigenvalues of $J_{\mathcal{F}}(P_{n}u_{*})$ have strictly negative
real part for $n\ge5$. Therefore, as it has already been established
in Figure \ref{fig:Time-evolution-of}, this steady state solution
is locally asymptotically stable in the sense of Corollary \ref{thm:asymptotic stability }.

\section{Concluding remarks}

Our primary motivation in this paper was to show that available numerical
tools in the literature can facilitate theoretical analysis of evolution
equations. Towards this end we developed a numerical framework for
theoretical analysis of evolution equations arising in population
dynamical models. The main idea behind this framework is to approximate
generators of semigroups of evolution equations and use numerical
tools to study stability of steady states of evolution equations.
We demonstrated the utility of our approach by applying the numerical
framework to both linear and nonlinear size-structured population
models. In particular, we generated the existence and stability regions
of the steady states of the both models (which can be difficult to
obtain by using conventional analytical tools. We showed that our
numerical framework can also be used to gain insight about the local
stability of stationary solutions. Furthermore, code used for the
numerical simulations can be obtained from our Github repository under
the open source MIT License (MIT) \citep{mirzaev2015steadystate}. 

Although the stability inequality in (\ref{eq:stability inequality})
holds for all finite time intervals, behavior of the solutions as
$t\to\infty$ is unclear. Hence, we note that the local stability
of the stationary solutions specified in Corollary \ref{thm:asymptotic stability }
is not in a usual Lyapunov sense. In order to prove Lyapunov stability
of stationary solutions using the approximation scheme presented in
Section \ref{sub:Approximation-scheme}, one has to prove uniform
convergence of the approximation scheme for all $t\ge0$. Hence, as
a future plan we wish to utilize the numerical framework presented
here to establish Lyapunov stability of stationary solutions of general
evolution equations.

\section*{Acknowledgements}

Funding for this research was supported in part by grants NSF-DMS
1225878 and NIH-NIGMS 2R01GM069438-06A2. This work utilized the Janus
supercomputer, which is supported by the National Science Foundation
(award number CNS-0821794) and the University of Colorado Boulder.
The Janus supercomputer is a joint effort of the University of Colorado
Boulder, the University of Colorado Denver and the National Center
for Atmospheric Research. 

\bibliographystyle{apalike}
\bibliography{numerical_equilibrium}

\end{document}